\documentclass[a4paper,12pt]{article}
\usepackage{amsmath}
\usepackage{amsfonts}
\usepackage{amssymb}
\usepackage{amsthm}
\usepackage{graphicx}
\usepackage{ifthen}
\usepackage{dsfont}
\usepackage{float}
\usepackage{verbatim}
\usepackage{xcolor}
\usepackage[numbers]{natbib}
\newtheorem{thm}{Theorem}[section]
\newtheorem{ex}[thm]{Example}
\newtheorem{coro}[thm]{Corollary}
\newtheorem{lem}[thm]{Lemma}
\newtheorem{defn}[thm]{Definition}
\newtheorem{rem}[thm]{Remark}
\newtheorem{prop}[thm]{Proposition}

\newcommand{\levy}{{L\'evy }}
\def\1{\mathds{1}}
\def\R{\mathbb{R}}
\def\P{\mathbb{P}}

\def\E{\mathbb{E}}
\def\N{\mathbb{N}}
\def\S{\mathbb{S}}
\def\d{\,\mathrm{d}}
\def\dd{\mathrm{d}}

\def\var{\mathrm{Var}}
\def\cov{\mathrm{Cov}}

\def\b{\beta}

\def\s{\sigma}

\def\law{\overset{\textnormal{law}}{=}}
\def\tp{\top}
\def\F{\mathcal{F}}
\def\H{\mathcal{H}}
\def\Z{\mathcal{Z}}

\def\t{\tau}

\begin{document}

\title{Modulated Information Flows in Financial Markets}
\author{Edward Hoyle$^1$, Andrea Macrina$^{2,\, 3
}$ and Levent Ali Meng\"ut\"urk$^{2}$ \\ \\
$^{1}$AHL Partners LLP, Man Group plc \\ London EC4R 3AD, U.K.\\
$^{2}$Department of Mathematics, University College London \\ London WC1E 6BT, U.K.\\
$^{3}$African Institute of Financial Markets \& Risk Management \\ University of Cape Town, Rondebosch 7701, RSA}
\maketitle

\begin{abstract}
We model continuous-time information flows generated by a number of information sources that switch on and off at random times. By modulating a multi-dimensional \levy random bridge over a random point field, our framework relates the discovery of relevant new information sources to jumps in conditional expectation martingales. In the canonical Brownian random bridge case, we show that the underlying measure-valued process follows jump-diffusion dynamics, where the jumps are governed by information switches. The dynamic representation gives rise to a set of stochastically-linked Brownian motions on random time intervals that capture evolving information states, as well as to a state-dependent stochastic volatility evolution with jumps. The nature of information flows usually exhibits complex behaviour, however, we maintain analytic tractability by introducing what we term the effective and complementary information processes, which dynamically incorporate active and inactive information, respectively. As an application, we price a financial vanilla option, which we prove is expressed by a weighted sum of option values based on the possible state configurations at expiry. This result may be viewed as an information-based analogue of Merton's option price, but where jump-diffusion arises endogenously. The proposed information flows also lend themselves to the quantification of asymmetric informational advantage among competitive agents, a feature we analyse by notions of information geometry.
\end{abstract}

\vspace{.25cm}
\noindent Corresponding author: a.macrina@ucl.ac.uk
\section{Introduction}
Information flows, especially in social and communication systems, usually exhibit complex structures and behaviours. Through time, not only new information channels may suddenly appear---e.g. following a new politician or a scientist in Twitter---but also some of the existing information sources may cease---e.g. an orbiting satellite stops operating or a national statistical agency stops reporting particular economic data series---and some information may only be temporarily interrupted over random periods of time. In general, if relevant information arrives gradually and continuously in time, then its impact on an observer's inference is small over short time periods. However, if important news arrives sporadically or new information sources appear at discrete points in time (e.g. a relevant data vendor becomes available to the system), then it is reasonable to expect that the impact of new information on the observer's inference is substantial. Such a view was taken in the work of \cite{19} for financial markets, where we find: {\it ``By its very nature, important information arrives only at discrete points in time. This component is modelled by a jump process reflecting the non-marginal impact of the information."} Recent examples of the new information source phenomenon in the financial markets are Elon Musk's tweet about taking Tesla private in 2018, including following announcements and developments, and the fall in the Swatch stock price when the Swiss National Bank stopped defending the EUR-CHF currency floor in January 2015. Another incident of a similar nature is the effect on Snapchat's share price following tweets by Kylie Jenner. A financial example for interrupted information is that of exchanges occasionally halting the trading of securities, following large price moves, with \emph{circuit breakers} for stocks and \emph{limit-up} pricing rules for commodity futures.

The general situation we keep in mind is one where observers or agents in a given system assess a signal on the basis of varying information sets, e.g. various news channels/feeds, accessible to them at any point in time. We develop a stochastic approach for information flows that can dynamically encapsulate information sources to switch on and off at random times through a modulated filtration. We show that intelligence, produced on the basis of such information flows, features explicit dependence on the number and the specific information channels that are active or inactive over the course of the inference period. While the approach we develop can be applied broadly in signal processing, stochastic filtering and related fields, our main aim is to systematically and explicitly relate randomly evolving information states to price dynamics, which provide a mathematically rich avenue of possible scenarios. Another area, where our framework may offer new insights, is the modelling of election dynamics where modulated information flows may impact the outcome of a democratic election. Here we refer to \cite{3a}.

We give an example in financial markets on price formation, which is a complicated mechanism. Market agents will buy or sell a given asset at a given price if suitably incentivised. For market makers, the incentive may not be to profit from anticipated price changes or income, but instead to earn the bid-ask spread or commissions. Although investors often seek to profit from anticipated price changes and income, they may have other incentives to trade that are not directly linked to these, including risk management (to increase, decrease or hedge risk), an evolving mandate (such as a tightening of environmental, social and governance criteria), a change in risk aversion, and a change in wealth.
However, it is commonly understood that the price of an asset should reflect the information that the market has about the asset. 
Consider the situation of a government considering the bailout of a bank or insurance company (the firm).
The market may be aware of the firm being in financial difficulty, and this would then be reflected in the prices of the firm's stock and bonds.
The market may also anticipate a government bailout.
However, due to insider trading legislation, any bailout negotiations between the government and the firm may be held behind closed doors.
The market cannot be sure of the bailout until an official government announcement.
We can view the government's communication of a bailout with the market as new source of information, which may have certain stylized effects on prices.
One such effect is an initial price correction as the market converts a bailout possibility into a certainty; a significant---and immediate---adjustment in the price.
The correction is the market absorbing the totality of the private bailout negotiations in one gulp.
Salient details for prices would include the bailout's anticipated size and duration.
A second effect is a series of small price adjustments as the government regularly reports on the progress of its bailout.
This may continue until either the firm recovers and no longer requires government assistance (loans the government has made the firm are repaid, equity the government has bought in the firm is sold), the firm is wholly taken over by the government, or the firm is allowed to fail.
At this point the source of market information on the firm ceases.

In order to explicitly model the flow of information, we introduce a multivariate process---each marginal being a \levy random bridge (LRB, see, \cite{15})---modulated over a random point field dynamically acting on its coordinates. LRBs form a large class of stochastic processes---where \levy processes are special cases---and the modulation allows for random activation and deactivation of the information flow. More general stochastic bridge processes is the class of randomized Markov bridges (RMB), which are constructed in \cite{18a}. As a canonical subclass, we focus on Brownian random bridges (BRBs) and prove a dynamical representation for the conditional expectation process under the proposed modulated filtration, which offers a new stochastic differential equation to model price dynamics under information switches. Let $X$ denote some square-integrable random variable (e.g. future payoff) and $(\boldsymbol{\xi}_t)_{t\in[0,T]}\in\R^n$ for some finite $T$ where each $(\xi^{(i)}_t)$ is an LRB with $\xi^{(i)}_{T}\law X$ for $i=1,\ldots,n$. Next, we choose some c\`adl\`ag jump process $(\boldsymbol{\tilde{J}}_t)_{t\in[0,T]}\in\{0,1\}^n$ associated to some random point field, and define the $\sigma$-algebra
\begin{align}
\F_t = \s\left( (\boldsymbol{\tilde{J}}_u \otimes \,\boldsymbol{\xi}_u)_{0\leq u \leq t}, (\boldsymbol{\tilde{J}}_u)_{0\leq u \leq t}  \right). \nonumber
\end{align}
If we introduce $\mathcal{J}_t = \{i : \tilde{J}_t^{(i)} = 1\}$, $C_t = \sum_{s\leq t} \1\{\boldsymbol{\tilde{J}}_{s}\ne\boldsymbol{\tilde{J}}_{s-}\}$ and $\pi_j = \inf\{t:C_t=j\}$, we have the following result, which we shall prove after later in the paper.
\begin{prop}
\label{ExpectationSDEinit}
Let $(\xi^{(i)}_t)$ be a BRB for $i=1,\ldots,n$. Then, the $(\F_t)$-martingale $(X_t)_{[0,T]}$ defined by $X_t:=\mathbb{E}[X \,|\, \F_t]$ admits the dynamic representation
\begin{align*}
	X_t &= \E[X] + \sum_{j=1}^{C_t}\left(\int_{\pi_{j-1}}^{\pi_j} \theta_s  \d W_s^{(j)}\right)\1\{\mathcal{J}_{\pi_{j-1}}\neq\emptyset \} +\left(\int_{\pi_{C_t}}^{t} \theta_s \d W_s^{(C_t+1)}\right)\1\{\mathcal{J}_{\pi_{C_t}}\neq\emptyset \} + \sum_{s\leq t}\Delta X_s,
\end{align*}
for $t<T$, where $(W_t^{(j)})$ for $j=1,\ldots,C_t+1$ is a one-dimensional standard $(\F_t)$-Brownian motion between jump times, $(\theta_t)$ is an $(\F_t)$-adapted process, and $\Delta X_t\neq 0$ if and only if $\mathcal{J}_t \backslash \mathcal{J}_{t-} \ne \emptyset$.
\end{prop}
Thus, a welcome consequence of our approach is an information-based {\it endogenous} jump-diffusion model with state-dependent stochastic volatility dynamics, arising naturally from the proposed information system. In this sense, we also extend \cite{4, 5, 6, 7, 16, 21}, which employ special properties of Brownian bridges (see \cite{12}, \cite{13}). We shall also show that the processes $(W_t^{(j)})$ above are stochastically-linked, where their dependence is dynamically controlled by $(\mathcal{J}_t)$. 

The modulation framework enables one to derive a Feynman-Ka\'c representation of the conditional expectation, gives an alternative expression for its jump-sizes, and extends Proposition \ref{ExpectationSDEinit} to the case of multiple random point fields. We associate these random point fields with projection-valued stochastic processes to construct a more general information system which may incorporate a wide range of complex behaviour, such as what we term \emph{information mixing}. By use of concepts in information geometry, we highlight how the impact of new information sources can be quantified from a geometrical perspective, which later allows for the modelling of information asymmetry. We also manage to provide an analytical expression for the price of a vanilla option under the modulated filtration, as an information-based analogue of the price obtained by \cite{19}, that admits a much broader class of random counting measures dictating price jumps.

We highlight the fact that we do not need to introduce jumps into the information flow by embedding discontinuous noise into individual information processes. 
In the BRB case, the jumps arise due to the discovery of \textit{new} information sources, even if BRBs have a \textit{continuous} state-space.
This has different and rather important implications on the dynamics of $(X_t)$, as opposed to including independent discontinuous noise with no information content on $X$. First, jumps are caused by random changes in the number of active information sources; hence, jumps carry information about $X$.
Second, the continuous part of $(X_t)$ is driven by different Brownian motions on random time intervals that capture the possible state-configurations of the information flow. Third, since the volatility process also jumps and is state-dependent, the framework offers a link to regime-switching models; see \cite{8, 14, 20}. Fourth, the undiscounted price process may remain constant for periods of time if all information sources are ``lost'', which perhaps could be viewed as a model for certain features arising in illiquid markets or circuit breakers.
Fifth, the proposed framework offers links to progressive enlargements of filtrations and stopped filtrations.

\section{Modulated information processes} \label{sec:framework}
Let $(\Omega, \mathcal{G}, (\mathcal{G}_t)_{t\geq0},\P)$ be a probability space equipped with a filtration $(\mathcal{G}_t)$. 
All considered filtrations are assumed to be right-continuous and complete. 
We introduce a random variable $X\in\mathcal{L}^2(\Omega,\mathcal{G},\P)$ with law $\nu$ and state-space $(\mathbb{X},\mathcal{B}(\mathbb{X}))$, where we assume $\mathbb{X}\subseteq\mathbb{R}$. 
We consider the time interval $\mathbb{T}=[0,1]$ --- and the set $\mathbb{T}_{*}=[0,1)$ --- though it is straightforward to consider a compact interval $[0,T]$ for $T<\infty$, instead. We let $(\boldsymbol{L}_t)_{t\geq0}$ be a multivariate \levy process taking values in $\R^n$ for $n\in\mathbb{N}_+$, with mutually independent coordinates (also independent of $X$) such that each $\{L^{(i)}_{t}\}_{i=1,\ldots,n}$ has density $0<f_t<\infty$ for all $t\in(0,1]$. We assume $\nu$ concentrates mass where $f_{1}$ is positive and finite $\nu$-a.s. Next we introduce a $(\mathcal{G}_t)$-adapted multivariate stochastic process $(\boldsymbol{\xi}_t)_{t\in\mathbb{T}}$ taking values in $\R^n$, where each marginal $(\xi^{(i)}_t)$ for $i=1,\ldots,n$ is a \levy random bridge (LRB) satisfying: (i) $\xi^{(i)}_{1}$ has marginal law $\nu$, (ii) for all $m\in\mathbb{N}_{+}$, every $0<t_{1}<\ldots<t_{m}<1$, every $(y_{1},\ldots,y_{m})\in\mathbb{R}^{m}$, and $\nu$-a.e. $x$,
\begin{align}
	\mathbb{P}(\xi^{(i)}_{t_{1}}\leq y_{1}^{(i)},\ldots, \xi^{(i)}_{t_{m}}\leq y_{m}^{(i)} | \xi^{(i)}_{1}=x)=\mathbb{P}(L^{(i)}_{t_{1}}\leq y_{1}^{(i)},\ldots, L^{(i)}_{t_{m}}\leq y_{m}^{(i)} | L^{(i)}_{1}=x), \nonumber
\end{align}
for $i=1,\ldots,n$. We denote the conditional measure as $g(x,y^{(i)},t)\dd y^{(i)} = \mathbb{P}(\xi^{(i)}_{t}\in \dd y^{(i)} | \xi^{(i)}_{1}=x)$ when it exists. The finite-dimensional distribution of each $(\xi^{(i)}_t)$ is
\begin{align}
\mathbb{P}(\xi^{(i)}_{t_{1}}\in \dd y_{1}^{(i)},\ldots, \xi^{(i)}_{t_{m}}\in \dd y_{m}^{(i)}, \xi^{(i)}_{1}\in \dd x)=\prod_{i=1}^m (f_{t_i-t_{i-1}}(y_i-y_{i-1})\dd y_i)\frac{f_{1-t_m}(x-y_m)}{f_1(x)}\nu(\dd x). \nonumber
\end{align}
The distributions of LRBs with discrete state-spaces can similarly be written using their probability mass functions. We note that LRBs are constructed in \cite{15} where, aside the development of the theory, applications in finance are discussed, in detail. In what follows, we use $(\boldsymbol{\xi}_t)$ to model information flows on $X$. Note that $X$ can be any square-integrable random variable and can be modelled (and simulated) independently from the chosen \levy process. For example, let $F:\mathcal{D}([0,1],\mathbb{R})\rightarrow\mathbb{R}$ be a functional from the Skorokhod space $\mathcal{D}([0,1],\mathbb{R})$ to the real line. Then, for some real-valued semimartingale $(S_t)$ adapted to $(\mathcal{G}_t)$, we can define $X=F((S_u);0\leq u \leq 1)$ as a model for a path-dependent signal. 

Next we introduce a mechanism for the activation and deactivation of information sources. Let $\zeta$ be a random point field (a point process) on $\mathbb{T}^n$, independent of $(\boldsymbol{\xi}_t)$, generating a collection of times $\{\varpi^{i}_{1}, \ldots,\varpi^{i}_{k_i}\}$ for $i\in\{1,\ldots,n\}$ and finite $k_i\in\mathbb{N}_0$. For every $i$, we associate the random sequence $\{\varpi^{i}_{1}, \ldots,\varpi^{i}_{k_i}\}$ to a coordinate of a $(\mathcal{G}_t)$-adapted c\`adl\`ag jump process $(\boldsymbol{\tilde{J}}_t)$ with state space $\S=\{0,1\}^n$. We then define the $\R^{n\times n}$-valued process $(\boldsymbol{J}_t)$ by
$J_t^{(i,j)} = \delta_{ij}\boldsymbol{\tilde{J}}_t$.
Thus, $(\boldsymbol{J}_t)$ is a diagonal matrix-valued process indicating which coordinates of $(\boldsymbol{\xi}_t)$ are active through the modulated information process $(\boldsymbol{J_t}\,\boldsymbol{\xi}_t)$. 
When the $i$th information source is inactive, the $i$th coordinate of $(\boldsymbol{J}_t\,\boldsymbol{\xi}_t)$ is identically zero.
We define a sub-algebra $\F_t\subseteq \mathcal{G}_t$ by
\begin{equation}\label{subF} 
\F_t = \s\left( (\boldsymbol{\tilde{J}}_u \otimes \,\boldsymbol{\xi}_u)_{0\leq u \leq t}, (\boldsymbol{\tilde{J}}_u)_{0\leq u \leq t}, X\1_{\{t=1\}}  \right), \nonumber
\end{equation} 
for $t\in\mathbb{T}$. Surely, we can write $\F_t = \s\left( (\boldsymbol{J}_u\,\boldsymbol{\xi}_u)_{0\leq u \leq t}, (\boldsymbol{J}_u)_{0\leq u \leq t}, X\1_{\{t=1\}}  \right)$. The reason why we diagonalize $(\boldsymbol{\tilde{J}}_t)$ will be clear later in Section \ref{modprojsec}. We also note that we add $X\1_{\{t=1\}}$ to the filtration to ensure that its value is revealed at $t=1$, even though all information may be inactive at $t=1$. This is not a mathematical requirement; if an envisaged application does not need $X$ to be (fully) observed at $t=1$, then $X\1_{\{t=1\}}$ may be excluded from the algebra $\mathcal{F}_t$. In the remainder of this paper, we shall omit it from the expressions unless necessary.
\begin{rem}
Define $\F^*_t = \s\left( (\boldsymbol{\tilde{J}}_u \otimes \,\boldsymbol{\xi}_u)_{0\leq u \leq t}, X\1_{\{t=1\}}  \right)$, hence, $\F_t^*\subseteq \mathcal{F}_t$. The $(\F_t)$-adapted switching process $(\boldsymbol{\tilde{J}}_t)$ is also $(\F^*_t)$-adapted if and only if $(\boldsymbol{\xi}_t)$ is continuous and $\P[\xi_t^{(i)}=0]=0$ for all $i\in\{1,\ldots, n\}$ for $t\in(0,1]$. 
\end{rem}
Hence, the individual appearance of $(\boldsymbol{\tilde{J}}_t)$ in the definition of $\F_t$ is not superfluous given that the framework allows the coordinate $(\xi^{(i)}_t)$ to be a jump process or such that $0< \P[\xi_t^{(i)}=0]<1$, for some $i\in\{1,\ldots, n\}$ and $t\in\mathbb{T}$. In these cases, it may not be possible to detect the (de)activation of a particular information coordinate if $(\boldsymbol{\tilde{J}}_t)$ is not $(\F_t)$-adapted.  

At time $t\in\mathbb{T}$, we denote the last time that the $i$th information source was active by $\t^{(i)}_t$.
That is, we define
\[ \t^{(i)}_t = 0 \vee \sup\{ u : J_u^{(i,i)} = 1, u\in[0,t]\}, \]
for $i=1,\ldots,n$, where we adopt the convention that $\sup \emptyset = -\infty$.
Thus, the process $(\t^{(i)}_t)$ is increasing and progressively measurable with respect to $(\F_t)$, and if the $i$th process has never been active up until time $t$, then $\t^{(i)}_t=0$. Also, by definition, the initial condition is $\t^{(i)}_0=0$ in either case, when $J_0^{(i,i)} = 0$ or $J_0^{(i,i)} = 1$. 
\begin{prop} \label{rem1}
The dynamics of $(\t^{(i)}_t)$ for $t\in\mathbb{T}$ are given by
\[ \dd \t^{(i)}_t = J_t^{(i,i)}\dd t + (t-\t^{(i)}_{t-}) \dd J^{(i,i)}_t. \]
\end{prop}
\begin{proof}
Per coordinate $i$, the dynamics of $(\tau^{(i)}_t)$ can be decomposed into its continuous and discontinuous parts, where $\dd J^{(i,i)}_t=0$ or $\dd J^{(i,i)}_t=1$, respectively. Given that $(J^{(i,i)}_t)$ has state-space $\{0,1\}$, we have $\tau^{(i)}_t=t$ whenever $J^{(i,i)}_t=1$ and $\tau^{(i)}_t<t$ whenever $J^{(i,i)}_t=0$. Hence, for the continuous part, when $J^{(i)}_t=1$, $\dd \tau^{(i)}_t=\dd t$, and when $J^{(i,i)}_t=0$, $\dd \tau^{(i)}_t=0$. As for the discontinuous part, just before a jump, if $J^{(i,i)}_{t-}=1$, then $\dd \tau^{(i)}_t=\dd t$ since $\tau^{(i)}_{t-}=t$. If just before a jump, $J^{(i,i)}_{t-}=0$, then $\dd \tau^{(i)}_t=(t-t^*)$ given that $\tau^{(i)}_{t-}=t^*$ for some $t^*<t$ and $\tau^{(i)}_t=t$.
\end{proof}
For our purposes, we shall prove that the conditional distribution of $X$ given $\F_t$ can be determined through time-changed LRBs.
\begin{prop} \label{lemmaMarkov}
Let $\xi^{(i)}(u)$ be the value of $(\xi^{(i)}_t)$ at $u\in\mathbb{T}$.
\begin{enumerate}
\item	For any $A\in\mathcal{B}(\mathbb{X})$ and $t_i \in \mathbb{T}_*$,
	\[ \P\left[X\in A \left|\, \{(\xi^{(i)}_s)_{0\le s\leq t_i}\}_{i=1,\ldots,n} \right.\right] 
			= \P\left[X\in A \left|\, \{\xi^{(i)}(t_i)\}_{i=1,\ldots,n} \right.\right]. \]
\item The sigma-algebra $\F_t$ for any $t\in\mathbb{T}$ is equal to 
\[ \F_t = \s\left( \{(\xi^{(i)}(\t^{(i)}_u))_{0\leq u \leq t}\}_{i=1,\ldots,n}, (\boldsymbol{J}_u)_{0\leq u \leq t} ,X\1_{\{t=1\}}\right).\]
\end{enumerate}
\end{prop}
\begin{proof} 
For the first part, for $t<1$, it is sufficient to show that
\begin{align*}
			\P\left[ X \in \d x_0 \left|\, \xi^{(1)}_{t_{1,1}} = x_{1,1}, \ldots, \xi^{(1)}_{t_{1,k_1}} = x_{1,k_1}, \ldots, \xi^{(n)}_{t_{n,1}} = x_{n,1}, \ldots, \xi^{(n)}_{t_{n,k_n}} = x_{n,k_n}\right.\right]&
			\\=\P\left[ X \in \d x_0 \left|\,  \xi^{(1)}_{t_{1,k_1}} = x_{1,k_1}, \ldots, \xi^{(n)}_{t_{n,k_n}} = x_{n,k_n}\right.\right]&,
\end{align*}
for all $k_i \in \N_+$ where $i\in\{1,\ldots,n\}$, all $0 < t_{i,1} < \cdots < t_{i,k_i} < 1$, and all $(x_{i,1}, \ldots,x_{i,k_i})\in \mathbb{R}^n$. We have
\begin{align*}
			&\P\left[ X \in \d x_0 \left|\, \xi^{(1)}_{t_{1,1}} = x_{1,1}, \ldots, \xi^{(1)}_{t_{1,k_1}} = x_{1,k_1}, \ldots, \xi^{(n)}_{t_{n,1}} = x_{n,1}, \ldots, \xi^{(n)}_{t_{n,k_n}} = x_{n,k_n}\right.\right]
			\\
				&=\frac{\P\left[ \left. \bigcap_{i=1}^{n} \bigcap_{j=1}^{k_i} \xi^{(i)}_{t_{i,j}} \in\dd x_{i,j} \,\right| X = x_0\right] \P[X\in \dd x_0]}
						{\int_{\mathbb{X}} \P\left[ \left. \bigcap_{i=1}^{n} \bigcap_{j=1}^{k_i} \xi^{(i)}_{t_{i,j}} \in\dd x_{i,j} \,\right| X = x_0\right] \P[X\in \dd x_0]}&
			\\
				&=\frac{\prod_{i=1}^{n} \P\left[ \left. \bigcap_{j=1}^{k_i} \xi^{(i)}_{t_{i,j}} \in\dd x_{i,j} \,\right| X = x_0\right] \P[X\in \dd x_0]}
						{\int_{\mathbb{X}} \prod_{i=1}^{n} \P\left[ \left. \bigcap_{j=1}^{k_i} \xi^{(i)}_{t_{i,j}} \in\dd x_{i,j} \,\right| X = x_0\right] \P[X\in \dd x_0]}.&
\end{align*}
Given $X$, each coordinate process $(\xi^{(i)}_t)$ is a \levy bridge. It follows that
	\begin{equation*} 
	\P\left[ \left. \bigcap_{j=1}^{k_i} \xi^{(i)}_{t_{i,j}} \in\dd x_{i,j} \,\right| X = x_0\right]
				= \frac{f_{1-t_{i,k_i}}(x_0-x_{i,k_i})}{f_1(x_0)} \prod_{j=1}^{k_i} f_{t_{i,j}-t_{i,j-1}}(x_{i,j}-x_{i,j-1})\dd x_{i,j},
	\end{equation*}
where $f_t(x)$ is the marginal density function of the underlying \levy process. 

Hence, we have the following:
	\begin{align*}
			&\P\left[ X \in \d x_0 \left|\, \xi^{(1)}_{t_{1,1}} = x_{1,1}, \ldots, \xi^{(1)}_{t_{1,k_1}} = x_{1,k_1}, \ldots, \xi^{(n)}_{t_{n,1}} = x_{n,1}, \ldots, \xi^{(n)}_{t_{n,k_n}} = x_{n,k_n}\right.\right]
			\\ 
					&\hspace{6cm}=\frac{\prod_{i=1}^n \frac{f_{1-t_{i,k_i}}(x_0-x_{i,k_i})}{f_1(x_0)}f_{t_{i,k_i}}(x_{i,k_i}) \, \nu(\dd x_0)}
								{\int_{\mathbb{X}} \prod_{i=1}^n \frac{f_{1-t_{i,k_i}}(x_0-x_{i,k_i})}{f_1(x_0)}f_{t_{i,k_i}}(x_{i,k_i}) \, \nu(\dd x_0)}&
			\\
					&\hspace{6cm}=\frac{\prod_{i=1}^n \P\left[ \left. \xi^{(i)}_{t_{i,k_i}} \in\dd x_{i,k_i} \,\right| X = x_0\right] \P[X\in \dd x_0]}
						{\int_{\mathbb{X}} \prod_{i=1}^n  \P\left[ \left. \xi^{(i)}_{t_{i,k_i}} \in\dd x_{i,k_i} \,\right| X = x_0\right] \P[X\in \dd x_0]}&
			\\ 
					&\hspace{6cm}=\frac{\P\left[ \left. \bigcap_{i=1}^{n} \xi^{(i)}_{t_{i,k_i}} \in\dd x_{i,k_i} \,\right| X = x_0\right] \P[X\in \dd x_0]}
						{\int_{\mathbb{X}} \P\left[ \left. \bigcap_{i=1}^{n}\xi^{(i)}_{t_{i,k_i}} \in\dd x_{i,k_i} \,\right| X = x_0\right] \P[X\in \dd x_0]}
			\\
					&\hspace{6cm}=\P\left[ X \in \d x \left|\,  \xi^{(1)}_{t_{1,k_1}} = x_{1,k_1}, \ldots, \xi^{(n)}_{t_{n,k_n}} = x_{n,k_n}\right.\right].&
	\end{align*}
For the second part, if $\P[\xi_t^{(i)}=0]=0$ for any $t\in(0,1]$, the coordinate $(J_t^{(i,i)} \xi^{(i)}_t)$ and $(\xi^{(i)}(\t^{(i)}_t))$ differ only when the $i$th source is inactive, during which the coordinate process is zero, and $(\xi^{(i)}(\t^{(i)}_t))$ takes the source's last active value. The statement holds directly for the complementary case where $\P[\xi_t^{(i)}=0]>0$ for any $t\in(0,1]$.
\end{proof}

\begin{prop} \label{conditionaldist}
The $\F_t$-conditional distribution of $X$ is given by
	\begin{align*}
		\P[X \in \dd x \,|\, \F_t] 
			 &= \frac{\prod_i g\left(x, \xi^{(i)}(\t^{(i)}_t), \t^{(i)}_t\right) \nu(\dd x)}
									{\int_{\mathbb{X}} \prod_i g\left(x, \xi^{(i)}(\t^{(i)}_t), \t^{(i)}_t \right) \nu(\dd x)},
	\end{align*}
	for $t \in \mathbb{T}_{*}$ given that $g(x,y,t)$ exists.
\end{prop}
\begin{proof}
	First we note that, for $t_i\in\mathbb{T}_{*}$ and $i=1,\ldots,n$, using the mutual independence of $L^{(i)}_t$s for $i=1,\ldots,n$, we have
	\begin{align} \label{Bayes}
		\P\left[X\in \dd x \left|\, \left\{\xi^{(i)}_{t_i}\right\}_{i=1,\ldots, n} \right.\right] = \frac{\prod^n_{i=1} g\left(x, \xi^{(i)}_{t_i}, t_i \right)  \nu(\dd x)}{\int_{\mathbb{X}} \prod^n_{i=1} g\left(x, \xi^{(i)}_{t_i}, t_i\right) \nu(\dd x)}.
	\end{align}
For the computation of the conditional distribution $\P[X \in \dd x \,|\, \F_t]$, the first step is to use $\{\xi^{(i)}(\t^{(i)}_t)\}_{i=1,\ldots, n}$ to enlarge the information set we are conditioning on, and then to apply the tower property. We here refer to Proposition \ref{lemmaMarkov}.
\begin{equation*}
		\P[X \in \dd x \,|\, \F_t] 
			 = \E \left[\left. \P\left[ X\in \dd x \left\vert\, \left\{(\xi^{(i)}_s)_{0\le s\leq \t^{(i)}_t}\right\}_{i=1,\ldots, n},\, \F_t \right.\right] \,\right| \F_t \right].
\end{equation*}
The $\sigma$-algebra $\F_t$ contains the history $(\boldsymbol{J}_s)_{0\le s\le t}$, which tells up to what time $t_i$ the information coordinates $\{\xi^{(i)}(\t^{(i)}_s)_{0\le s\le t_i}\}_{i=1,\dots,n}$ have been active. Thus once the stopping times $\{\t^{(i)}_t\}_{i=1,\dots,n}$ have occurred, one knows that $(\xi^{(i)}(\t^{(i)}_t))=(\xi^{(i)}_s)_{0\le s\le t_i}$ for $\tau^{(i)}_t=t_i\le t$. Therefore we may apply Proposition \ref{lemmaMarkov} to obtain 
\begin{equation*}			 	
\E \left[\left. \P\left[ X\in \dd x \left\vert\, \left\{(\xi^{(i)}_s)_{0\le s\leq \t^{(i)}_t}\right\}_{i=1,\ldots, n},\, \F_t \right.\right] \,\right| \F_t \right]= \E \left[\left. \P\left[ X\in \dd x \left|\, \left\{\xi^{(i)}(\t^{(i)}_t) \right\}_{i=1,\ldots, n},\, \F_t \right.\right] \,\right| \F_t \right].
\end{equation*}
By use of Equation (\ref{Bayes}) and $\F_t$-measurability, it follows that
\begin{align*}
		\P[X \in \dd x \,|\, \F_t]&= \E \left[\left. \P\left[ X\in \dd x \left|\, \left\{\xi^{(i)}(\t^{(i)}_t) \right\}_{i=1,\ldots, n},\, \F_t \right.\right] \,\right| \F_t \right]
		\\ &= \frac{\prod^n_{i=1} g\left(x, \xi^{(i)}(\t^{(i)}_t), \t^{(i)}_t\right) \nu(\dd x)}
									{\int_{\mathbb{X}} \prod^n_{i=1} g\left(x, \xi^{(i)}(\t^{(i)}_t), \t^{(i)}_t\right) \nu(\dd x)},
	\end{align*}	
for $t_i\in\mathbb{T}_{*}$ and $i=1,\ldots,n$, which gives the statement.
\end{proof}
We can view the appearance of new sources of information from a Hilbert space perspective. We let $\langle \, . \, \rangle$ be the inner product on the Hilbert space of square-integrable functions $\mathcal{L}^{2}(\mathbb{G})$ on a measurable set $\mathbb{G}\subset\mathbb{R}^{2}$, such that $v,w\in\mathcal{L}^{2}(\mathbb{G})$ are orthogonal if $\langle v,w \rangle =0$, and where
\begin{equation}
\label{eq:ell2rep}
\mathcal{L}^{2}(\mathbb{G})=\bigoplus^{n+1}_{i=1}\mathcal{L}^{2}_{i}(\mathbb{G}), \notag
\end{equation}
where $\mathcal{L}^{2}_{i}(\mathbb{G})$ and $\mathcal{L}^{2}_{j}(\mathbb{G})$ are mutually orthogonal closed subspaces of $\mathcal{L}^{2}(\mathbb{G})$ for $i\ne j$, such that any function in $\mathcal{L}^{2}(\mathbb{G})$ can uniquely be represented by the sum of its projections onto the subspaces $\mathcal{L}^{2}_{i}(\mathbb{G})$ for $i=1,\ldots,n+1$ that span $\mathcal{L}^{2}(\mathbb{G})$.

As an example, we choose $\zeta$ on $\mathbb{T}^n$ such that $k_i=1$ for all $i=1,\ldots,n$ and the set of random times is reduced to $\{\varpi^{i}_{1}:i=1,\ldots,n\}$, where we have an ordered collection $0 < \varpi^{\pi(1)}_{1} < \ldots < \varpi^{\pi(i)}_{1} < \ldots < \varpi^{\pi(n)}_{1} < \infty$, given that $\pi$ is the permutation on the coordinates of $(\boldsymbol{\tilde{J}}_t)$ that produces this order. Letting $H_{\alpha}(t)=\boldsymbol{1}(\alpha \leq t)$, and denoting $\varpi^{\pi(0)}_{1}=0$, we define an $n+1$ vector $\{\boldsymbol{I}\}_{0\leq t \leq 1}$ by
\begin{align}
\boldsymbol{I}_t = \left[1-H_{\varpi^{\pi(1)}_{1}}(t), \ldots, 
H_{\varpi^{\pi(i-1)}_{1}}(t)(1-H_{\varpi^{\pi(i)}_{1}}(t)), \ldots,
H_{\varpi^{\pi(n)}_{1}}(t)\right]^{\tp}. \notag
\end{align}
We assume that $X$ has density, where we write $p(x \,|\, \F_t )\dd x = \P[X \in \dd x \,|\, \F_t]$ for $0\leq t < 1$. Note that if $\hat{p}_{t}(x)=\sqrt{p(x \,|\, \F_t )}$ for $0\leq t < 1$, then $\hat{p}\in\mathcal{L}^{2}(\mathbb{X}\times \mathbb{T}_{*})$. Also, we let $p_t^{(0)}=\sqrt{p(x)}$ and the sequence
\begin{equation}
\label{eq:squarerootdensity}
p_{t}^{(i)}(x)=\sqrt{p\left( x \,|\, \xi^{(\pi(1))}_{t},\, ... \, , \,\xi^{(\pi(i))}_{t}\right)}, \notag
\end{equation}
for $i=1,...,n$. Hence, $p_{t}^{(1)}(x)=p\left( x \,|\, \xi^{(\pi(1))}_{t}\right)^{\frac{1}{2}}$, $p_{t}^{(2)}(x)=p\left( x \,|\, \xi^{(\pi(1))}_{t},\xi^{(\pi(2))}_{t}\right)^{\frac{1}{2}}$ etc., and $p^{(i)}\in\mathcal{L}^{2}(\mathbb{X}\times \mathbb{T}_{*})$ for $i=0,\ldots,n$. We let the disjoint sets $\mathbb{T}_{i}$, for $i=0,...,n$ be such that $\mathbb{T}_{0}=\{t \in \mathbb{T}_* : t<\varpi^{\pi(1)}_{1}\}$, $\mathbb{T}_{i}=\{t \in \mathbb{T}_* : \varpi^{\pi(i)}_{1}\leq t<\varpi^{\pi(i+1)}_{1}\}$ for $i=1,...,n-1$, and $\mathbb{T}_{n}=\{t \in \mathbb{T}_* : \varpi^{\pi(n)}_{1}\leq t\}$. 
Finally, for the next statement, we define the measurable function $\widehat{\phi}^{(i)}\in\mathcal{L}^{2}(\mathbb{X}\times \mathbb{T}_{*})$ for $i=0,...,n$ by
\begin{equation}
\label{eq:includingfourier}
\widehat{\phi}^{(i)}_{t}(x) = 
\begin{cases}
p_{t}^{(i)}(x) & \textrm{if $t \in \mathbb{T}_{i}$},\\
0 & \textrm{otherwise.} \notag \\
\end{cases}
\end{equation}
\begin{prop}
Let $\hat{p}_{t}(x)=p(x \,|\, \F_t )^{\frac{1}{2}}$ as defined above. Then,
\begin{align}
\hat{p}=\widehat{\phi}^{(0)}+ \cdots + \widehat{\phi}^{(n)}.  \notag
\end{align}
Hence, $p(x \,|\, \F_t ) = (p_t^{(0)}(x))^2\boldsymbol{I}_t^{(1)}+ \cdots + (p_t^{(n)}(x))^2\boldsymbol{I}_t^{(n+1)}$. 
\end{prop}
\begin{proof}
Using Proposition \ref{lemmaMarkov} and Proposition \ref{conditionaldist}, we can rewrite $\widehat{\phi}^{(i)}\in\mathcal{L}^{2}(\mathbb{X}\times \mathbb{T}_{*})$ as
\begin{equation}
\widehat{\phi}^{(i)}_{t}(x) = 
\begin{cases}
\sqrt{p(x \,|\, \F_t )} & \textrm{if $t \in \mathbb{T}_{i}$},\\
0 & \textrm{otherwise.} \notag \\
\end{cases}
\end{equation}
for $i=0,...,n$. Thus, having $\mathbb{T}_*= \bigcup^{n}_{i=0} \mathbb{T}_{i}$, and $\mathbb{G}=\mathbb{X} \times \mathbb{T}_{*}$ as the domain of the measurable functions $\hat{p}$ and $p^{(i)}$ for $i=0,\ldots,n$, we can consider the orthogonal decomposition
\begin{align}
\mathcal{L}^{2}(\mathbb{X}\times \mathbb{T}_{*})=\bigoplus^{n+1}_{i=1}\mathcal{L}^{2}_{i}(\mathbb{X}\times\mathbb{T}_{*}), \notag
\end{align}
where $\widehat{\phi}^{(i)}\in\mathcal{L}^{2}_{i+1}(\mathbb{X}\times\mathbb{T}_*)$ for $i=0,\ldots,n$. Hence, $\langle \widehat{\phi}^{(i)},\widehat{\phi}^{(j)} \rangle =0$ for $i\ne j$ on $\mathbb{X}\times\mathbb{T}_*$, and we have the first part of the statement. The second part follows directly since $\mathbf{I}_{t}^{(i)}\mathbf{I}_{t}^{(j)}=0$ for $i\ne j$ and $\mathbf{I}_{t}^{(i)}\mathbf{I}_{t}^{(i)}=\mathbf{I}_{t}(i)$ for $i=1,\ldots,n+1$.
\end{proof}
Since each $\phi^{(i)}_{t}(x)$ takes values in $\mathbb{R}$ for $i=0,\ldots,n$ and $x\in\mathbb{X}$, we can as well work with any Hilbert space $\mathcal{H}$ isomorphic to $\mathbb{R}^{n+1}$, and canonically represent $\hat{p}$ as an ($n+1$)-tuple in $\mathbb{R}^{n+1}$, and write $\mathbf{I}(i)=e_{i}=\left[0, \ldots,  1, \ldots, 0\right]^{\tp}$ for $i=1,\ldots,n+1$. 
Thus, for $\mathcal{H}\cong\mathbb{R}^{n+1}$, since $e_{i}\in\mathcal{H}$ forms a complete orthonormal sequence for $i=1,\ldots,n+1$, we have
\begin{align}
\hat{p}&=\sum^{n+1}_{i=1}\langle \hat{p},e_{i} \rangle e_{i}=\sum^{n+1}_{i=1}p^{(i-1)}e_{i}, \notag
\end{align}
and hence, $p^{(i)}$'s for $i=0,\ldots,n$ are the Fourier coefficients of $\hat{p}$. The insight gained from the Hilbert space brings forth a geometrical interpretation. The function $\hat{p}$ is a non-negative function, and for a fixed time $t$, the integral of the square of $\hat{p}_{t}$ on $\mathbb{X}$ is unity. Thus, $\hat{p}_{t}$ determines a point on the positive orthant of the unit sphere $\mathcal{S}^{+}\subset\mathcal{L}^{2}$, where the geodesics between $p^{(i)}$s and $p^{(j)}$s for $i\ne j$ are the circles on the Riemannian manifold ($\mathcal{S}^{+},\langle \, . \, \rangle)$. Note that here, $p^{(i)}$ and $p^{(j)}$ are defined in terms of $i$ and $j$ number of active information sources, respectively.
\begin{rem} \label{thm:remgeometry1}
The non-marginal impact of a new information source can be measured by the spherical distance between the Fourier coefficients $p^{(i)}$ and $p^{(i+1)}$ for $i=0,\ldots,n-1$ on $\mathcal{S}^{+}$.
\end{rem}

\subsection{Endogenous jump-diffusion}\label{EJD}
We focus on a canonical case where $(L^{(i)}_t)$ is a Brownian motion such that $L^{(i)}_t  \sim \mathcal{N}(0,\s_i^2 t)$ for some $\s_i>0$ for $i=1,\ldots,n$. The reason is two-fold: (i) the Brownian case provides valuable analytical tractability in deriving dynamic representations, (ii) we explicitly show how conditional expectation martingales may exhibit jumps even if $(\boldsymbol{\xi}_t)$ is a continuous process.

\begin{coro}
Let $(\boldsymbol{L}_t)_{0\le t}$ be a multivariate Brownian motion such that $L^{(i)}_t  \sim \mathcal{N}(0,\s_i^2 t)$ for $\s_i>0$. Then,
	\begin{align*}
		\P[X \in \dd x \,|\, \F_t] 
			 &= \frac{\prod_i h\left(x, \xi^{(i)}(\t^{(i)}_t), \t^{(i)}_t, \s_i\right) \nu(\dd x)}
									{\int_{\mathbb{X}} \prod_i h\left(x, \xi^{(i)}(\t^{(i)}_t), \t^{(i)}_t, \s_i\right) \nu(\dd x)},
	\end{align*}
	for $t \in \mathbb{T}_{*}$, where the transformation $h:\R\times\R\times\mathbb{T}_{*}\times\R_+ \rightarrow \R_+$ is the Gaussian map given by
$ h(x, y, t, \s) = \exp\left\{(x y-t x^2/2) / (\s^2(1-t))\right\}.$ 
\end{coro}
\begin{proof}
Since $L^{(i)}_t  \sim \mathcal{N}(0,\s_i^2 t)$ is a Brownian motion for some $\s_i>0$, we can write $ \xi^{(i)}_t \law t X + \s_i \b_t^{(i)}$ for $i=1,\ldots,n$, where $(\beta_t^{(i)})_{0\le t\le 1}$ is a standard Brownian bridge; a Gaussian process with mean function identically zero, and covariance function $(s,t)\mapsto \min[s,t] - st$, for $s,t\in\mathbb{T}$. This anticipative representation holds since (i) $\xi^{(i)}_{1}$ has marginal law $\nu$, and (ii) for all $m\in\mathbb{N}_{+}$, every $0<t_{1}<\ldots<t_{m}<1$, every $(y_{1},\ldots,y_{m})\in\mathbb{R}^{m}$, and $\nu$-a.e. $x$, $\mathbb{P}(\xi^{(i)}_{t_{1}}\in \dd y_{1}^{(i)},\ldots, \xi^{(i)}_{t_{m}}\in \dd y_{m}^{(i)} | \xi^{(i)}_{1}=x)=\mathbb{P}(L^{(i)}_{t_{1}}\in \dd y_{1}^{(i)},\ldots, L^{(i)}_{t_{m}}\in \dd y_{m}^{(i)} | L^{(i)}_{1}=x)$. Given that $X$ and the Brownian bridges $\{(\b_t^{(i)})\}_{i=1,\ldots,n}$, are mutually independent,	
	\begin{align} \label{Bayesbrownian}
		\P\left[X\in \dd x \left|\, \left\{\xi^{(i)}_{t_i}\right\}_{i=1,\ldots, n} \right.\right]
			&= \frac{\prod^n_{i=1} \exp\left[-\frac{1}{2} \left( \frac{\xi^{(i)}_{t_i} - t_i x}{\s_i\sqrt{t_i(1-t_i)}} \right)^2  \right] \nu(\dd x)}
							{\int_{\mathbb{X}} \prod^n_{i=1} \exp\left[-\frac{1}{2} \left( \frac{\xi^{(i)}_{t_i} - t_i x}{\s_i\sqrt{t_i(1-t_i)}} \right)^2  \right] \nu(\dd x)}\nonumber
			\\ &= \frac{\prod^n_{i=1} \exp\left( \frac{x \xi^{(i)}_{t_i}-t_i x^2/2}{\s_i^2 (1-t_i)}\right) \nu(\dd x)}
							{\int_{\mathbb{X}} \prod^n_{i=1} \exp\left(\frac{ x \xi^{(i)}_{t_i}-t_i x^2/2}{\s_i^2 (1-t_i)}\right) \nu(\dd x)}. \nonumber
	\end{align}
The rest of the proof follows that of Proposition \ref{conditionaldist}.
\end{proof}

Accordingly, unless stated otherwise, we shall set $ \xi^{(i)}_t \law t X + \s_i \b_t^{(i)}$ for $i=1,\ldots,n$. For $t\in(0,1)$, $\xi_t^{(i)}/t$ is equal to $X$ plus some independent Gaussian noise with variance $\s_i^2(1-t)/t$, and $\xi_1^{(i)}= X$, $i=1,\ldots,n$. This functional form is also natural from the standpoint of stochastic filtering. 
If there were linear dependence (with non-singular covariation) between $\{(\b_t^{(i)})\}_{i=1,\ldots,n}$, then there would exist a linear transformation of $(\boldsymbol{\xi}_t)$ that would fit the framework. 
Hence, allowing linear dependence between $\{(\b_t^{(i)})\}_{i=1,\ldots,n}$ does not enrich the model. 

Working with $(\boldsymbol{\tilde{J}}_t \otimes \,\boldsymbol{\xi}_t)$, we maintain analytic tractability by constructing what we term \emph{effective} and \emph{complementary} information processes. For the randomly evolving active coordinates of $(\boldsymbol{\tilde{J}}_t \otimes \,\boldsymbol{\xi}_t)$, this parametrisation allows to reduce the multi-dimensional information system to a dynamically constructed one-dimensional (effective) information process containing all active information at any given time. This idea reduces the complexity of the information flow model substantially without compromising its effectiveness or reducing its richness.

\begin{defn}
Let the set-valued stochastic process $(\mathcal{J}_t)$ be given by $\mathcal{J}_t = \{i : J_t^{(i,i)} = 1\}$.
For $\mathcal{J}_t$ non-empty, the effective information process $(\hat{\xi_t})_{t\in\mathbb{T}}$ on $\R$ is defined by
\[ \hat{\xi_t}=\hat{\s_{t}}^2 \sum_{i \in \mathcal{J}_t} \s^{-2}_i\xi^{(i)}_t\]
with effective volatility parameter
$ \hat{\s_{t}} = \left(\sum_{i \in \mathcal{J}_t} \s^{-2}_i\right)^{-1/2}$.
For $\mathcal{J}_t = \emptyset$, set $\hat{\xi}_t = 0$ and $\hat{\s_{t}}=0$.
\end{defn}
Adding elements to $\mathcal{J}_t$ decreases the size of $\hat{\s_{t}}$, hence, adding sources of information decreases the effective noise in the system.
The dynamics of the effective information process are dependent on the current state $\boldsymbol{\tilde{J}}_t\in\S$.
We can make this dependence more explicit by rewriting the effective information as follows:
\[ \hat{\xi}_t = 
		\left\{ 
				\begin{aligned} &0 && \text{if $\boldsymbol{J}_t=\mathbf{0}$}, 
					\\ &\frac{\boldsymbol{\xi}_t^{\tp} \boldsymbol{J}_t \, \boldsymbol{\rho}}{ \mathbf{1}^{\tp} \boldsymbol{J}_t \, \boldsymbol{\rho}} && \text{otherwise,} 
				\end{aligned} 
		\right. 
\]
where $\boldsymbol{\rho} =(\s_1^{-2},\ldots,\s_n^{-2})^{\tp}$ and $\mathbf{1}=(1,\ldots,1)^{\tp}$.  
\begin{lem} \label{lemmaeffectiveinfo}
	If $\mathcal{J}_t\neq \emptyset$, then the effective information process is given by
	\[ \hat{\xi}_t \law tX + \hat{\s_{t}} \hat{\b}_t, \]
	where $(\hat{\b}_t)_{t\in\mathbb{T}}$, defined by
	$ \hat{\b}_t = \hat{\s_{t}} \sum_{i \in \mathcal{J}_t} \s^{-1}_i\b^{(i)}_t$,
	is a standard Brownian bridge between the jumps of $(\mathcal{J}_t)$.
\end{lem}
\begin{proof}
	By definition, for $\mathcal{J}_t\neq\emptyset$, the effective information process is $\hat{\xi}_t = tX + \hat{\s_{t}} \sum_{i \in \mathcal{J}_t} \s^{-1}_i\b^{(i)}_t$.
	A linear combination of independent Brownian bridges is a Brownian bridge.
	Further, we have
\begin{align*}	
\E\left[\hat{\s_{t}} \sum_{i \in \mathcal{J}_t} \s^{-1}_i\b^{(i)}_t\right] =0, \hspace{0.1in} \var\left[\hat{\s_{t}} \sum_{i \in \mathcal{J}_t} \s^{-1}_i\b^{(i)}_t\right] =t(1-t),
\end{align*}
	hence, $(\hat{\b}_t)$ is a \emph{standard} Brownian bridge between the jumps of $(\mathcal{J}_t)$. 
\end{proof}
The effective information process jumps every time $(\boldsymbol{\tilde{J}}_t)$ changes state. The jump is caused by the change in the number of Brownian bridges defining the effective information process as well as the number of volatility control parameters defining the effective volatility process $(\hat{\s_{t}})$.

\begin{defn}
Let the set-valued complementary process $(\mathcal{J}_t^\complement)$ be given by $\mathcal{J}_t^\complement = \{i : J_t^{(i,i)} = 0\}$. The complementary information process $(\eta_t)_{t\in\mathbb{T}}$ is a function-valued process defined by
\begin{equation*}
\eta_t: x\mapsto \prod_{i\in\mathcal{J}_t^\complement} h\left(x, \xi^{(i)}(\t^{(i)}_t), \t^{(i)}_t, \s_i\right),
\end{equation*}
where $\eta_t:=1$ if $\mathcal{J}_t^\complement = \emptyset$.
\end{defn}
The complementary information process is piecewise constant between state changes of $(\boldsymbol{\tilde{J}}_t)$.
The next statement on the joint Markov property of effective and complementary information processes will be very useful. For the rest of this work, we define the measure-valued process $(\nu_t)_{t\in \mathbb{T}_{*}}$ by $\nu_t(A)=\P[X \in A \,|\, \F_t]$ for $A\in\mathcal{B}(\mathbb{X})$.
\begin{prop} \label{effectivecomplementary}
	The measure $\nu_t(A)$ satisfies 
	\begin{align*}
	\nu_t(A) = \P[X \in A \,|\, \hat{\xi}_t, \eta_t],
	\end{align*}
for $t \in \mathbb{T}_{*}$ and for any $A\in\mathcal{B}(\mathbb{X})$.
\end{prop}
\begin{proof}
	Using Proposition \ref{conditionaldist} and noting that $\t^{(i)}_t = t$ if $i\in\mathcal{J}_t$, we have
	\begin{align*}
		 \P[X \in \dd x \,|\, \F_t]=& \frac{\prod_{i\in\mathcal{J}_t} h\left(x, \xi^{(i)}(\t^{(i)}_t), \t^{(i)}_t, \s_i\right)
									\prod_{i\in\mathcal{J}_t^\complement} h\left(x, \xi^{(i)}(\t^{(i)}_t), \t^{(i)}_t, \s_i\right) \nu(\dd x)}
									{\int_{\mathbb{X}} \prod_{i\in\mathcal{J}_t} h\left(x, \xi^{(i)}(\t^{(i)}_t), \t^{(i)}_t, \s_i\right)
											\prod_{i\in\mathcal{J}_t^\complement} h\left(x, \xi^{(i)}(\t^{(i)}_t), \t^{(i)}_t, \s_i\right) \nu(\dd x)}		
								\\ =& \frac{\prod_{i\in\mathcal{J}_t} \exp\left\{\frac{x \xi^{(i)}(\t^{(i)}_t)-\t^{(i)}_t x^2/2}{\s_i^2(1-\t^{(i)}_t)}\right\}	
									\eta_t(x) \, \nu(\dd x)}
									{\int_{\mathbb{X}} \prod_{i\in\mathcal{J}_t} \exp\left\{\frac{x \xi^{(i)}(\t^{(i)}_t)-\t^{(i)}_t x^2/2}{\s_i^2(1-\t^{(i)}_t)}\right\}	
											\eta_t(x) \, \nu(\dd x)}
						\\ =& \frac{\exp\left\{\sum_{i\in\mathcal{J}_t}\frac{x \xi^{(i)}_t-t x^2/2}{\s_i^2(1-t)}\right\}	
									\eta_t(x) \, \nu(\dd x)}
									{\int_{\mathbb{X}} \exp\left\{\sum_{i\in\mathcal{J}_t}\frac{x \xi^{(i)}_t-t x^2/2}{\s_i^2(1-t)}\right\}	
											\eta_t(x) \, \nu(\dd x)}
					   \\=& \frac{h\left(x, \hat{\xi}_t, t, \hat{\s_{t}} \right) \eta_t(x) \, \nu(\dd x)}
									{\int_{\mathbb{X}} h\left(x, \hat{\xi}_t, t, \hat{\s_{t}} \right) \eta_t(x) \,\nu(\dd x)},
	\end{align*}
	where $h\left(x, 0, t, 0\right):=1$, and where we adopt the convention that a product with an empty index set is equal to one.
\end{proof}

We shall now derive a dynamical equation for the conditional expectation process, which we later use to model asset price dynamics, as an example. The dynamics of the conditional expectation process turns out to be jump-diffusion. The jumps arise from the activation of new sources of information generating the filtration. For the evolution of $(\nu_t)$, we introduce two $(\F_t)$-adapted counting processes $(C_t)_{t\in\mathbb{T}}$ and $(N_t)_{t\in\mathbb{T}}$, given by
\begin{align*}
&C_t = \sum_{s\leq t} \1\{\boldsymbol{\tilde{J}}_{s}\ne\boldsymbol{\tilde{J}}_{s-}\},
&N_t = \sum_{s\leq t} \delta_s, \hspace{0.1in} \text{where} \hspace{0.1in} \delta_s = \1\{\mathcal{J}_s \backslash \mathcal{J}_{s-} \ne \emptyset \}.
\end{align*}
Hence, $C_t$ is the number of times $(\boldsymbol{\tilde{J}}_t)$ has changed state up to and including time $t$, and $N_t$ is the number of state changes in which at least one inactive information process becomes active. In view of Proposition \ref{measureSDE}, we define the following.
\begin{defn} Let $\pi_j = \inf\{t:C_t=j\}$,
with $\pi_0=0$. The process $(M_t)_{t\in\mathbb{T}_*}$ is defined by
\[ M_t= \sum_{j=1}^{C_t+1} M_t^{(j)} + \int_0^t (\hat{\xi}_s-\hat{\xi}_{s-}) \d N_s, \]
where $(M_t^{(j)})_{t\in\mathbb{T}_*}$ is given by
\[
	M_t^{(j)} = 
	\left\{ \begin{aligned}
		&0, && \text{for $t < \pi_{j-1}$,}
		\\ &(\hat{\xi}_{t\wedge \pi_{j}-}-\hat{\xi}_{\pi_{j-1}}) - \int_{\pi_{j-1}}^{t\wedge \pi_{j}-}\frac{\E[ X \,|\, \hat{\xi}_s, \eta_s ] - \hat{\xi}_s}{ 1-s} \d s, &&\text{for $\mathcal{J}_{\pi_{j-1}}\neq\emptyset$,}
		\\ &0, && \text{for $\mathcal{J}_{\pi_{j-1}}=\emptyset$.}
	\end{aligned} \right.
\]
\end{defn}
The following statement is important in order to assign a sufficient structure for the dynamics of the continuous part of $(M_t)$, so that the stochastic integrals with respect to the processes $\{(M_t^{(i)})\}_{i=1,\ldots,n}$ are well-defined for the subsequent proposition.
\begin{prop}\label{piecewisemartingale}
	Let $t\in[\pi_{j-1},\pi_{j}\wedge 1)$, where $\pi_{j-1}$ and $\pi_{j}$ are two consecutive jump times of $(C_t)$ where $\mathcal{J}_{\pi_{j-1}}\neq\emptyset$. Then, $(M_t)_{\pi_{j-1} \leq t < \pi_{j}}$ is an $(\F_t)$-martingale.
\end{prop}
\begin{proof}
The integrability condition $\E[\,|M_{t}|\,]<\infty$ for $t\in\mathbb{T}_*$ is satisfied. Next we show $\E[M_u\,\vert\,\F_t]=M_t$ for $u\geq t$, where we consider the random interval $(t,u)\in[\pi_{j-1},\pi_j\wedge 1)$, with $\mathcal{J}_{\pi_{j-1}}\neq\emptyset$, over which $(M_t)$ has no discontinuity. If $\mathcal{J}_{\pi_{j-1}}=\emptyset$, then $(M_t)_{\pi_{j-1} \leq t < \pi_{j}}$ would simply be 0.
	First, note that we have
	\begin{align}
		\E[M_{u} \,|\, \F_t]&= \E[M_{u} - M_{t} \,|\, \F_t] + M_{t} \nonumber \\
		&=M_{t}+\E\left[\left. \hat{\xi}_u-\hat{\xi}_t \,\right| \hat{\xi}_t, \eta_t \right] +\E\left[\left.\int^{u}_{t}\frac{\hat{\xi}_s}{1-s} \dd s \,\right| \hat{\xi}_t, \eta_t \right] -\E\left[\left. \int^{u}_{t}\frac{\E[X \,|\, \hat{\xi}_s, \eta_s]}{1-s} \dd s \,\right| \hat{\xi}_t, \eta_t \right]. \nonumber
	\end{align}
Writing the terms explicitly and using the tower property, we have
	\begin{align}
		\E[M_{u} \,|\, \F_t] = M_{t} &+ \E\left[\left. Xu+\hat{\s}_{u}\hat{\b}_u \,\right| \hat{\xi}_t, \eta_t  \right]-\E\left[\left.Xt+\hat{\s}_{t}\hat{\b}_t \,\right| \hat{\xi}_t, \eta_t  \right] \nonumber \\
			&+\E\left[X \left|\, \hat{\xi}_t, \eta_t \right.\right]\int^{u}_{t}\frac{s}{1-s}\dd s +\E\left[\left. \int^{u}_{t}\frac{\hat{\s}_{s}\hat{\b}_s}{1-s}\dd s \,\right| \hat{\xi}_t, \eta_t \right] \nonumber \\
			&-\E\left[X \left|\, \hat{\xi}_t, \eta_t \right.\right]\int^{u}_{t}\frac{1}{1-s}\dd s. \nonumber
	\end{align}
	Note that all the terms involving $X$ disappear, and we are left with
	\begin{align}
		\E[M_{u} \,|\, \F_t] &= M_{t}+\E\left[\left. \hat{\s}_{u}\hat{\b}_u \,\right| \hat{\xi}_t, \eta_t \right]
				-\E\left[\left. \hat{\s}_{t}\hat{\b}_t \,\right| \hat{\xi}_t, \eta_t \right] 
				+\int^{u}_{t}\frac{\E[\hat{\s}_{s}\hat{\b}_s \,|\, \hat{\xi}_t, \eta_t]}{1-s}\dd s \nonumber \\
		&= M_{t} + \hat{\s_{t}}\left(\E\left[\left. \hat{\b}_u \,\right| \hat{\xi}_t, \eta_t \right]
						-\E\left[\left. \hat{\b}_t \,\right| \hat{\xi}_t, \eta_t \right] +\int^{u}_{t}\frac{\E[\hat{\b}_s \,|\, \hat{\xi}_t, \eta_t]}{1-s}\dd s\right), \nonumber 
	\end{align}
	where we used the fact that $(\hat{\s_{s}})_{t\leq s \leq u}=\hat{\s_{t}}$, since we reside in $[\pi_{j-1},\pi_j\wedge 1)$, and hence, $\hat{\s_t}$ remains constant.
	By recalling the mutual independence between $X$ and all the $\{(\b^{(i)}_t)\}_{i=1\ldots,n}$, and the tower property, we may write the following:
	\begin{align}
		\E\left[\left. \hat{\b}_u \,\right| \hat{\xi}_t, \eta_t  \right]=\E\left[\left. \E\left[\left. \hat{\b}_u \,\right| X,\hat{\b}_t,\eta_t\right] \,\right| \hat{\xi}_t, \eta_t  \right]=\E\left[\left. \E\left[\left. \hat{\b}_u \,\right|\hat{\b}_t \right] \,\right| \hat{\xi}_t, \eta_t\right]=\frac{1-u}{1-t}\,\E\left[\left. \hat{\b}_t \,\right| \hat{\xi}_t, \eta_t\right].\nonumber
	\end{align}
	With the above expression at hand, we have
	\begin{align}
		\E\left[\left. \hat{\b}_u \,\right| \hat{\xi}_t, \eta_t \right]
			-\E\left[\left. \hat{\b}_t \,\right| \hat{\xi}_t, \eta_t \right] 
			+\int^{u}_{t}\frac{\E[\hat{\b}_s \,|\, \hat{\xi}_t, \eta_t]}{1-s}\dd s=0, \nonumber
	\end{align}
which proves $\E[M_{u} \,|\, \F_t]= M_{t}$ for $(t,u)\in[\pi_{j-1},\pi_j\wedge 1)$.
\end{proof}

\begin{prop} \label{measureSDE}
The measure-valued process $(\nu_t)$ satisfies
\begin{align*}
\nu_t(A) &= \nu(A) + \sum_{j=1}^{C_t}\left(\int_{\pi_{j-1}}^{\pi_j}\int_A \frac{x-\E[X\,|\, \hat{\xi}_s, \eta_s]}{\hat{\s_{s}}^2 (1-s)} \nu_s(\dd x) \d M_s^{(j)}\right)\1\{\mathcal{J}_{\pi_{j-1}}\neq\emptyset \} \\ 
&+\left(\int_{\pi_{C_t}}^{t}\int_A\frac{x-\E[X\,|\, \hat{\xi}_s, \eta_s]}{\hat{\s_{s}}^2 (1-s)} \nu_s(\dd x) \d M_s^{(C_t+1)}\right)\1\{\mathcal{J}_{\pi_{C_t}}\neq\emptyset \}
+ \sum_{s\leq t} (\nu_s(A) - \nu_{s-}(A))\delta_s,
\end{align*}
for $t \in \mathbb{T}_{*}$ and for any $A\in\mathcal{B}(\mathbb{X})$.
\end{prop}
\begin{proof}
Since $\zeta$ on $\mathbb{T}^n$ has finite number of jumps, $(\nu_t)$ can be represented by the sum of the continuous and the discontinuous components via the decomposition
$\nu_t(A) = \nu_t^c(A) + \sum_{s\leq t} \Delta \nu_s(A)$. 
For the continuous part $\nu_t^c(\dd x)$,
\begin{align}
	\nu_t^c(\dd x) = \frac{h\left(x, \hat{\xi}_t^c, t, \hat{\s}_{t}^c \right) \eta_t^c(x) \, \nu(\dd x)}
									{\int_{\mathbb{X}} h\left(x, \hat{\xi}_t^c, t, \hat{\s}_{t}^c \right) \eta_t^c(x) \,\nu(\dd x)}. \nonumber
\end{align}
If $\mathcal{J}_{t}=\emptyset$, then $h\left(x, 0, t, 0 \right)=1$ and $\dd\nu_t^c(\dd x)=0$,  thus we consider time periods where $\mathcal{J}_{t}\neq\emptyset$. As such, the continuous part of volatility $(\hat{\s}_{t}^c)$ is constant between discontinuities and satisfies $(\hat{\s}_{t}^c)>0$. 
Then, we define a function of $\hat{\xi}_t^c$, $\eta_t^c$ and $t$ as
\[ g\left( \hat{\xi}_t^c, t, \eta_t^c ;x,\hat{\s}_{t}^c,\dd x\right) = h\left(x,\hat{\xi}_t^c, t, \hat{\s}_{t}^c \right) \eta_t^c(x) \, \nu(\dd x), \]
and also its integral over $\mathbb{X}$, where
\[ G\left( \hat{\xi}_t^c, t, \eta_t^c ;\hat{\s}_{t}^c\right)=\int_{\mathbb{X}} g\left( \hat{\xi}_t^c, t, \eta_t^c ;x,\hat{\s}_{t}^c,\dd x\right).\]
Using Ito's lemma, we have
\begin{align}
	\dd g\left( \hat{\xi}_t^c, t, \eta_t^c ;x,\hat{\s}_{t}^c,\dd x\right)
		&= \frac{\partial g}{\partial t}\dd t + \frac{\partial g}{\partial \hat{\xi}_t^c} \dd \hat{\xi}_t^c + \frac{\partial^2 g}{2 \partial (\hat{\xi}_t^c)^2}\dd (\hat{\xi}_t^c)^2 \nonumber
	\\
	&\hspace{1.105cm}+ \frac{\partial g}{\partial \eta_t^c}\dd \eta_t^c + \frac{\partial^2 g}{2\partial (\eta_t^c)^2}\dd (\eta_t^c)^2 + \frac{\partial^2 g}{\partial \hat{\xi}_t^c \eta_t^c}\dd \hat{\xi}_t^c \dd \eta_t^c \nonumber
	\\
	&= \frac{\partial g}{\partial t}\dd t + \frac{\partial g}{\partial \hat{\xi}_t^c} \dd \hat{\xi}_t^c + \frac{\partial^2 g}{2 \partial (\hat{\xi}_t^c)^2} (\hat{\s}_{t}^c)^2 \dd t  \nonumber
	\\
	&= g\left( \hat{\xi}_t^c, t, \eta_t^c ;x,\hat{\s}_{t}^c,\dd x\right) \left(\frac{x\hat{\xi}_t^c}{(\hat{\s}_{t}^c)^2(1-t)^2}\dd t + \frac{x}{(\hat{\s}_{t}^c)^2(1-t)}\dd \hat{\xi}_t^c \right), \nonumber
\end{align}
since the quadratic variation of $(\dd \hat{\b}_t^c)$ is $\dd t$ and $\eta_t^c$ (the continuous part of $\eta_t$) is constant between discontinuities.  
Then using Fubini's theorem, we can write
\begin{align}
	\dd G\left( \hat{\xi}_t^c, t, \eta_t^c ;\hat{\s}_{t}^c\right) &= \int_{\mathbb{X}} g\left( \hat{\xi}_t^c, t, \eta_t^c ;x,\hat{\s}_{t}^c,\dd x\right) \left(\frac{x\hat{\xi}_t^c}{(\hat{\s}_{t}^c)^2(1-t)^2}\dd t + \frac{x}{(\hat{\s}_{t}^c)^2(1-t)}\dd \hat{\xi}_t^c \right) \nonumber
	\\
	&= G\left( \hat{\xi}_t^c, t, \eta_t^c ;\hat{\s}_{t}^c\right) \left(\frac{\E[X\,|\, \hat{\xi}_t^c, \eta_t^c]\hat{\xi}_t^c}{(\hat{\s}_{t}^c)^2(1-t)^2}\dd t + \frac{\E[X\,|\, \hat{\xi}_t^c, \eta_t^c]}{(\hat{\s}_{t}^c)^2(1-t)}\dd \hat{\xi}_t^c \right). \nonumber
\end{align}
Finally, for the bracket $\left\langle G,G \right\rangle$, we have
\[\dd \left\langle G,G \right\rangle = G\left( \hat{\xi}_t^c, t, \eta_t^c ;\hat{\s}_{t}^c\right)^2\frac{\E[X\,|\, \hat{\xi}_t^c, \eta_t^c]^2}{(\hat{\s}_{t}^c)^2(1-t)^2}\dd t, \]
and the bracket $\left\langle g,G \right\rangle$ satisfies
\[\dd \left\langle g,G \right\rangle = g\left( \hat{\xi}_t^c, t, \eta_t^c ;x,\hat{\s}_{t}^c,\dd x\right)G\left( \hat{\xi}_t^c, t, \eta_t^c ;\hat{\s}_{t}^c\right)\frac{x\E[X\,|\, \hat{\xi}_t^c, \eta_t^c]}{(\hat{\s}_{t}^c)^2(1-t)^2}\dd t. \]
Using the quotient rule and rearranging terms, we have
\[\dd \nu_t^c(\dd x) = \frac{x - \E[X\,|\, \hat{\xi}^c_t, \eta^c_t]}{(\hat{\s}_{t}^c)^2(1-t)}\nu_t^c(\dd x) \dd M_t^c, \]
for $t \in \mathbb{T}_{*}$ and $\mathcal{J}_{t}\neq\emptyset$, where 
\[\dd M_t^c = \dd \hat{\xi}^c_t - \frac{\E[ X \,|\, \hat{\xi}^c_t, \eta^c_t ] - \hat{\xi}^c_t}{1-t} \d t. \]
The integral in the continuous part of $(M_t)$ is well-defined over $[0,t]$ with respect to the Lebesgue measure. That is, having $\pi^*_j\in[0,t)$ for $j=0,\ldots,m<\infty$ as a subset of $\pi_j$'s where $\mathcal{J}_{\pi_{j-1}}\neq\emptyset$ for $j\geq1$, and denoting $u_{m+1}=\pi^*_{m+1}-=t$,
\begin{eqnarray}
\int_0^t\frac{\E[ X \,|\, \hat{\xi}_s, \eta_s ] - \hat{\xi}_s}{ 1-s} \d s &=&\sum_{j=1}^m \lim_{u_{j-1} \rightarrow (\pi^*_{j-1})+} \hspace{0.1in} \lim_{u_j \rightarrow \pi^*_j-} \int_{u_{j-1}}^{u_{j}}\frac{\E[ X \,|\, \hat{\xi}_s, \eta_s ] - \hat{\xi}_s}{ 1-s} \d s \nonumber
\\
&+&\sum_{j=1}^m \lim_{u_j \rightarrow \pi^*_j+} \hspace{0.1in} \lim_{u_{j+1} \rightarrow (\pi^*_{j+1})-} \int_{u_{j}}^{u_{j+1}}\frac{\E[ X \,|\, \hat{\xi}_s, \eta_s ] - \hat{\xi}_s}{1-s} \d s. \nonumber
\end{eqnarray}
Finally, the discontinuous part can be rewritten as
\begin{align}
\sum_{s\leq t} \Delta \nu_s(\dd x) 
= \sum_{s\leq t}\left(\nu_s(\dd x)-\nu_{s-}(\dd x)\right)\1\{\mathcal{J}_s \backslash \mathcal{J}_{s-} \ne \emptyset \}, \notag
\end{align}
for $t \in \mathbb{T}_{*}$. The statement follows by Lebesgue integration over any $A\in\mathcal{B}(\mathbb{X})$.
\end{proof}

The dynamics of $(M_t)_{t\in\mathbb{T}_*}$ are state-dependent, adapting to the system's information configuration at any given time. The process $(\nu_t)_{t\in\mathbb{T}_*}$ evolves continuously when there is no information switch or when an active source becomes inactive. The dynamics exhibit jumps from one state-dependent martingale to another only when an information source becomes active.
The process stays constant if all information sources are inactive over a random period of time.

\begin{prop} \label{ExpectationSDE}
Let $(W_t)$ be an $(\F_t)$-adapted process defined by $W_t=M_t/\hat{\s}_t$ when $\hat{\s}_t>0$.
\begin{enumerate}
\item Let $t\in[\pi_{j-1},\pi_{j}\wedge 1)$, where $\pi_{j-1}$ and $\pi_{j}$ are two consecutive jump times of $(C_t)$ where $\mathcal{J}_{\pi_{j-1}}\neq\emptyset$. Then $(W_t)_{\pi_{j-1} \leq t < \pi_{j}}$ is an $(\F_t)$-Brownian motion.
\item Let $(W_t^{(j)})$ be the Brownian motion given above for $t\in[\pi_{j-1}, \pi_j \wedge 1)$, where $\mathcal{J}_{\pi_{j-1}}\neq\emptyset$. Then, the $(\F_t)$-martingale $(X_t)_{t\in\mathbb{T}}$ defined by $X_t:=\mathbb{E}[X \,|\, \F_t]$, satisfies
\begin{align*}
	X_t &= \E[X] + \sum_{j=1}^{C_t}\left(\int_{\pi_{j-1}}^{\pi_j} \frac{\Gamma_s}{\hat{\s_{s}} (1-s)}  \d W_s^{(j)}\right)\1\{\mathcal{J}_{\pi_{j-1}}\neq\emptyset \} \\
	&+\left(\int_{\pi_{C_t}}^{t} \frac{\Gamma_s}{\hat{\s_{s}} (1-s)} \d W_s^{(C_t+1)}\right)\1\{\mathcal{J}_{\pi_{C_t}}\neq\emptyset \} + \sum_{s\leq t} (X_s - X_{s-})\delta_s,
\end{align*}
for $t\in\mathbb{T_*}$, where $(\Gamma_t)_{t\in\mathbb{T}_*}$ given by $\Gamma_t =\var[X\,|\, \hat{\xi}_t, \eta_t]$ is an $(\F_t)$-supermartingale.
\item More generally, the $(\F_t)$-martingale $(X_t^{(k)})_{t\in\mathbb{T}}$ defined by $X_t^{(k)}:=\mathbb{E}[X^k \,|\, \F_t]$ for any $k\geq 1$ satisfies
\begin{align*}
	X_t^{(k)} &= \E[X^k] + \sum_{j=1}^{C_t}\left(\int_{\pi_{j-1}}^{\pi_j} \frac{X_s^{(k+1)} - X_s^{(k)}X_s}{\hat{\s_{s}} (1-s)}  \d W_s^{(j)}\right)\1\{\mathcal{J}_{\pi_{j-1}}\neq\emptyset \} \notag \\ 
	&+\left(\int_{\pi_{C_t}}^{t} \frac{X_s^{(k+1)} - X_s^{(k)}X_s}{\hat{\s_{s}} (1-s)} \d W_s^{(C_t+1)}\right)\1\{\mathcal{J}_{\pi_{C_t}}\neq\emptyset \} + \sum_{s\leq t} (X^{(k)}_s - X^{(k)}_{s-})\delta_s,
\end{align*}
for $t\in\mathbb{T}_{*}$, given that $X\in\mathcal{L}^{k+1}(\Omega,\mathcal{G},\P)$.
\end{enumerate}
\end{prop}
\begin{proof}
The random time $\pi_{j-1}$ is $(\F_t)$-measurable, $\hat{\s}_t>0$ for $t\in[\pi_{j-1},\pi_{j}\wedge 1)$ whenever $\mathcal{J}_{\pi_{j-1}}\neq\emptyset$, $W_{\pi_{j-1}}=0$ and the subprocess $(\hat{\b}_t - \hat{\b}_{\pi_{j-1}})_{\pi_{j-1}\leq t < \pi_j}$ is an $(\F_t)$-Brownian bridge. Thus, the bracket $\dd\left\langle W_t, W_t \right\rangle$ for $t\in[\pi_{j-1},\pi_{j}\wedge1)$ is $\dd t$ given $(\F_t)$. Since the paths $(M_t)_{\pi_{j-1}\leq t < \pi_{j}}$ and $(\hat{\s}_t)_{\pi_{j-1}\leq t < \pi_{j}}$ are continuous, $(W_t)_{\pi_{j-1}\leq t < \pi_{j}}$ is an $(\F_t)$-Brownian motion by L\'evy's characterization.
The dynamics given in the second part follows directly from Proposition \ref{measureSDE}, the $(\hat{\s}_t)$-standardization of $(M_t)$ and Lebesgue integration. For the $(\F_t)$-supermartingale property of $(\Gamma_t)$, define $(S_t)_{0\leq t < 1}$ by $S_t = X_t^2$. Using Ito's lemma, $(S_t)$ is an $(\F_t)$-submartingale. Then from Doob-Meyer decomposition,
\begin{align*}
\E\left[ \E[X^2 \,|\, \F_t] - S_t \,|\, \F_s \right] &= \E\left[ \E[X^2 \,|\, \F_t] - \left(Y_t - I_t\right) \,|\, \F_s \right] \notag \\
&\leq \var[X\,|\, \F_s],
\end{align*} 
where $(Y_t)$ is an $(\F_t)$-martingale and $(I_t)$ is an increasing predictable process. The final part is given by the following:
\begin{align*}
	X_t^{(k)} &= \E[X^k] + \sum_{j=1}^{C_t}\left(\int_{\pi_{j-1}}^{\pi_j} \frac{\E[X^{k+1}\,|\, \hat{\xi}_s, \eta_s] - \E[X^{k}\,|\, \hat{\xi}_s, \eta_s]X_s}{\hat{\s_{s}} (1-s)}  \d W_s^{(j)} \right)\1\{\mathcal{J}_{\pi_{j-1}}\neq\emptyset \}\notag \\ 
	&+\left(\int_{\pi_{C_t}}^{t} \frac{\E[X^{k+1}\,|\, \hat{\xi}_s, \eta_s] - \E[X^{k}\,|\, \hat{\xi}_s, \eta_s]X_s}{\hat{\s_{s}} (1-s)} \d W_s^{(C_t+1)}\right)\1\{\mathcal{J}_{\pi_{C_t}}\neq\emptyset \} + \sum_{s\leq t} (X^{(k)}_s - X^{(k)}_{s-})\delta_s,
\end{align*}
which follows from the first two parts and by using Lebesgue integration on Proposition \ref{measureSDE}.
\end{proof}
The endogenous nature of the jump-diffusion, resulting from the modulation of information flow, emerges from the behaviour of $(\F_t)$. The jumps are a direct result of the activation of information coordinates determined by $(\boldsymbol{\tilde{J}}_t \otimes \,\boldsymbol{\xi}_t)$. This is different from generating jump diffusion by directly specifying the process itself as the sum of drifted Brownian motion and a compound Poisson process, see \cite{9}, for example.
\begin{coro}
The system of $(\F_t)$-Brownian motions $(W_t^{(j)})$'s for $j=1,\ldots,C_t+1$ are stochastically-linked through $(\mathcal{J}_t)$.
\end{coro}
Proposition \ref{ExpectationSDE} proves Proposition \ref{ExpectationSDEinit}. The diffusion coefficient $(\theta_t)$ in Proposition \ref{ExpectationSDEinit} can be interpreted as a stochastic volatility process with jumps. This process arises naturally from the system, which is a welcome consequence of the proposed framework without an {\it a priori} assumption on volatility dynamics. In applications, chiefly in asset pricing and financial risk management, the volatility process---and its dynamical equation---are of crucial importance. Accurate estimates of volatilities are important for measuring the risk of financial assets.
\begin{prop}\label{jumpdist}
Let $\mathcal{K}_t = \mathcal{J}_t \backslash \mathcal{J}_{t-}$.	The jump size of $(X_t)$ at time $t$ is
	$X_t-X_{t-} = g(Z) - X_{t-}$,
	where the conditionally normal random variable $Z$ is given by
	\[ Z = \sum_{i\in\mathcal{K}_t} \frac{ \xi_t^{(i)}}{\s_i^2(1-t)}, \]
	and the function $g:\R\rightarrow\R$, for $t \in\mathbb{T}_{*}$, is given by
	\[ g(z) = \frac{\int_{\mathbb{X}} x \exp(xz)\prod_{i \notin \mathcal{K}_t} h\left(x, \xi^{(i)}(\t^{(i)}_t), \t^{(i)}_t, \s_i\right) 
							\prod_{i \in \mathcal{K}_t} \exp\left( \frac{-t x^2}{2\s_i^2(1-t)} \right) \nu(\dd x)}
									{\int_{\mathbb{X}} \exp(xz)\prod_{i \notin \mathcal{K}_t} h\left(x, \xi^{(i)}(\t^{(i)}_t), \t^{(i)}_t, \s_i\right) 
													\prod_{i \in \mathcal{K}_t} \exp\left( \frac{-t x^2}{2\s_i^2(1-t)} \right) \nu(\dd x)}. \]
\end{prop}
\begin{proof}
	For $i\in\mathcal{K}_t$, we have the following
	\[ \left.\xi_t^{(i)} \right|_{X, \F_{t-}} \sim \mathcal{N}\left(\xi^{(i)}(\t^{(i)}_{t-})+(t-\t^{(i)}_{t-})X, \,\s_i^2(t-\tau^{(i)}_{t-})(1-t) \right). \]
	Given $X$, the variables $\{\xi_t^{(i)}\}_{i=1,\ldots,n}$ are mutually independent.
	It follows that the conditional distribution of $Z$ is also Gaussian, that is, 
	\[ \left. Z \right|_{X, \F_{t-}} \sim \mathcal{N}\left(\sum_{i \in \mathcal{K}_t} \frac{\xi^{(i)}(\t^{(i)}_{t-})+(t-\t^{(i)}_{t-})X}{\s_i^2(1-t)}, \, 
					\sum_{i \in \mathcal{K}_t}\frac{t-\tau^{(i)}_{t-}}{\s_i^2(1-t)} \right). \]
	Hence, the density of $Z$ is given by
	\[ z \mapsto \frac{1}{\sqrt{2\pi V}}\int_{x\in\R} \exp\left( -\frac{(z-U(x))^2}{2V} \right) \nu_{t-}(\dd x), \]
	where we have defined the following:
	\begin{align*}
		U(x) = \sum_{i \in \mathcal{K}_t} \frac{\xi^{(i)}(\t^{(i)}_{t-})+(t-\t^{(i)}_{t-})x}{\s_i^2(1-t)}, \hspace{0.1in}
		V = \sum_{i \in \mathcal{K}_t}\frac{t-\tau^{(i)}_{t-}}{\s_i^2(1-t)}.
	\end{align*}
Note that one can decompose the conditional distribution of $X$ in terms of $\mathcal{K}_t$ and write
	\[ 	\P[X \in \dd x \,|\, \F_t] 
				= \frac{\exp(xZ)\prod_{i \notin \mathcal{K}_t} h\left(x, \xi^{(i)}(\t^{(i)}_t), \t^{(i)}_t, \s_i\right) 
							\prod_{i \in \mathcal{K}_t} \exp\left(\frac{-t x^2}{2\s_i^2(1-t)} \right) \nu(\dd x)}
									{\int_{\mathbb{X}} \exp(xZ)\prod_{i \notin \mathcal{K}_t} h\left(x, \xi^{(i)}(\t^{(i)}_t), \t^{(i)}_t, \s_i\right) 
													\prod_{i \in \mathcal{K}_t} \exp\left(\frac{-t x^2}{2\s_i^2(1-t)} \right)\nu(\dd x)}. \]	
Hence, the statement follows.
\end{proof}
For the next statement, for $t<1$, we denote $\mu(\hat{\xi}_t, \eta_t,t) =(\E[ X \,|\, \hat{\xi}_t, \eta_t ] - \hat{\xi}_t)/(1-t)$ when $\mathcal{J}_{t}\neq\emptyset$, and $\mu(\hat{\xi}_t, \eta_t,t) = 0$ otherwise. We also recall $\hat{\s}_t=0$ when $\mathcal{J}_{t}=\emptyset$.
\begin{prop}\label{feynmankac}
Let $(\lambda_t)$ be the intensity process of $(N_t)$, and $\psi:\R\rightarrow\R$ and $\phi:\R\rightarrow\R$ be continuous bounded functions.
\begin{enumerate}
\item	The functional $\E\left[\left. e^{-\int_t^1 \psi(s)\dd s}\phi(X) \,\right| \F_t\right]$, hence,
	\[v(\hat{\xi}_t, \eta_t,t) := \E\left[\left. e^{-\int_t^1 \psi(s)\dd s}\phi(X) \,\right| \hat{\xi}_t, \eta_t,X\1_{\{t=1\}}\right], \]
	satisfies the partial differential equation (PDE)
\begin{multline*}
		\frac{\partial v(\hat{\xi}_t, \eta_t,t)}{\partial t} 
			+ \mu(\hat{\xi}_t, \eta_t,t)\frac{\partial v(\hat{\xi}_t, \eta_t,t)}{\partial \hat{\xi}_t} 
			+ \hat{\s}_t^2\frac{\partial^2 v(\hat{\xi}_t, \eta_t,t)}{2\partial \hat{\xi}_t^2}
		\\ 
			- \psi(t)v(\hat{\xi}_t, \eta_t,t) 
			+ (v(\hat{\xi}_t, \eta_t,t) 
			- v(\hat{\xi}_{t-}, \eta_{t-},t-))\lambda_t = 0,
\end{multline*}
with the boundary condition $v(\hat{\xi}_1, \eta_1,1) = \phi(X)$.
\item Choose the random point field $\zeta$ such that $\tilde{J}_t^{(i)}=1$ implies $\tilde{J}_u^{(i)}=1$ for $u\geq t$. Let there be at least one active source of information at $t=0$. Then,
	\[v(\hat{\xi}_t, t) = \E\left[\exp\left(-\int_t^1 \psi(s)\dd s\right)\phi(X) \,\bigg\vert\, \hat{\xi}_t\right], \]
	satisfies the same PDE as $v(\hat{\xi}_t, \eta_t,t)$ with the boundary condition $v(\hat{\xi}_1,1) = \phi(X)$.
\end{enumerate}
\end{prop}
\begin{proof}
	For part one, by use of the Doob-Meyer decomposition,  we can write 
	$N_t = \hat{N}_t + \Lambda_t$,
	where $(\hat{N}_t)$ is an $(\F_t)$-adapted martingale and $\Lambda_t = \int_0^t \lambda_s \dd s$ is the compensator process.
	Transforming the conditional expectation $v(\hat{\xi}_t, \eta_t,t)$ by multiplying it with the corresponding exponential function, we define the martingale $\hat{v}(\hat{\xi}_t, \eta_t,t) = e^{-\int_0^t \psi(s)\dd s} v(\hat{\xi}_t, \eta_t,t)$. 
	Then, by applying Ito's lemma to $\hat{v}(\hat{\xi}_t, \eta_t,t)$ after decomposing it into continuous and discontinuous parts, we have
\begin{align*}
&\frac{\hat{v}(\hat{\xi}_t, \eta_t,t)}{\partial t}\dd t + \frac{\hat{v}(\hat{\xi}_t, \eta_t,t)}{\partial \hat{\xi}_t}\left(\dd M_t + \mu(\hat{\xi}_t, \eta_t,t)\dd t\right) \\
&+ \hat{\s}_t^2\frac{\partial^2 \hat{v}(\hat{\xi}_t, \eta_t,t)}{2\partial \hat{\xi}_t^2}\dd t + \left(\hat{v}(\hat{\xi}_t, \eta_t,t) - \hat{v}(\hat{\xi}_{t-}, \eta_{t-},t-) \right)(\dd\hat{N}_t + \dd\Lambda_t).
\end{align*} 	
Note that for the continuous part, we have $\dd \eta_t = 0$ and $\dd \left\langle \eta_t,\eta_t \right\rangle = 0$. Thus, once the conditional expectation is taken with respect to $\F_t$, the continuous part of $(M_t)$ and the discontinuous part $(\hat{N}_t)$ vanish, 
\begin{multline*}
		e^{-\int_0^t \psi(s)\dd s}\left(\frac{\partial v(\hat{\xi}_t, \eta_t,t)}{\partial t} - \psi(t)v(\hat{\xi}_t, \eta_t,t) + \mu(\hat{\xi}_t, \eta_t,t)\frac{\partial v(\hat{\xi}_t, \eta_t,t)}{\partial \hat{\xi}_t} + \hat{\s}_t^2\frac{\partial^2 v(\hat{\xi}_t, \eta_t,t)}{2\partial \hat{\xi}_t^2}\right)\dd t \\
			+ e^{-\int_0^t \psi(s)\dd s}(v(\hat{\xi}_t, \eta_t,t) 
			- e^{-\int_0^{t-} \psi(s)\dd s}v(\hat{\xi}_{t-}, \eta_{t-},t-))\lambda_t  \dd t = 0,
\end{multline*}			
Once we divide both sides by $e^{-\int_0^t \psi(s)\dd s}\dd t$, the jump part remains to be 
\begin{align*}
\left(v(\hat{\xi}_t, \eta_t,t) - e^{\int_{t-}^{t} \psi(s)\dd s}v(\hat{\xi}_{t-}, \eta_{t-},t-)\right)\lambda_t,
\end{align*}
However, $\int_{t-}^{t} \psi(s)\dd s=0$ due to continuity, and hence, $e^{\int_{t-}^{t} \psi(s)\dd s}=1$.
	Finally, $X\1_{\{t=1\}}$ ensures that the boundary condition is satisfied even if $\mathcal{J}_1=\emptyset$.
	
For part two, active sources never deactivate. Thus for inactive states $i\in\mathcal{J}_t^c$, the complementary information is $\eta_t=1$ for all $t\in[0,1]$, since $\xi^{(i)}(0)=0$ and $\tau_t^{(i)}=0$ must hold for all $t\in[0,1]$. Therefore, we can write $\E\left[\exp\left(-\int_t^1 \psi(s)\dd s\right)\phi(X) \,\bigg\vert\, \F_t \right] =v(\hat{\xi}_t, t)$.  
	The boundary condition follows since at least one coordinate of $(\boldsymbol{J}_t\,\boldsymbol{\xi}_t)$ is always active over $\mathbb{T}$.
\end{proof}

\subsection{Multiple point fields}
As an extension to \cite{4}, we consider a system where the random variable $X$ is given by a function of multiple factors about which observers have differing information. That is, for some $m\in\mathbb{N}_+$, we introduce a vector $\boldsymbol{X}$ of mutually independent random variables $X^\alpha\in\mathcal{L}^2(\Omega,\mathcal{G},\P)$ with law $\nu^\alpha$ for $\alpha=1,\ldots,m$, respectively, each with the state-space $(\mathbb{X},\mathcal{B}(\mathbb{X}))$, where $\mathbb{X}\subseteq\mathbb{R}$. We treat the case where $X=g(X^1,\ldots,X^m)$ for some bounded measurable function $g:\mathbb{X}^m\rightarrow\mathbb{X}$.

Accordingly, we introduce $\R^{n(\alpha)}$-valued $(\mathcal{G}_t)$-adapted multivariate stochastic processes $(\boldsymbol{\xi}^\alpha_t)_{t\in\mathbb{T}}$, for $\alpha=1,\ldots,m$, where $n(\alpha)$ highlights that the dimensions of $(\boldsymbol{\xi}^\alpha_t)$s may be different across $\alpha$. 
As before, we define information coordinates by $\xi^{\alpha,(i)}_t \law t X^\alpha + \s_i^\alpha \b_t^{\alpha,(i)}$,
for $\s_i^\alpha>0$ and $i=1,\ldots,n(\alpha)$. For parsimony, we assume that the standard Brownian bridges $\{\b_t^{\alpha,(i)}\}$s are mutually independent and of $X^\alpha$s across $i=1,\ldots,n(\alpha)$ and $\alpha=1,\ldots,m$.

Next we introduce a set of mutually independent random point fields $\zeta^\alpha$ for $\alpha=1,\ldots,m$ on $\mathbb{T}^{n(\alpha)}$, which are also independent of $(\boldsymbol{\xi}^\alpha_t)$s, each generating a collection of times $\{\varpi^{\alpha,i}_{1}, \ldots,\varpi^{\alpha,i}_{k_i}\}$ for $i=1,\ldots,n(\alpha)$ and finite $k_i\in\mathbb{N}_0$. For every $(\alpha,i)$, we associate the random sequence $\{\varpi^{\alpha,i}_{1}, \ldots,\varpi^{\alpha,i}_{k_i}\}$ to a coordinate of a $(\mathcal{G}_t)$-adapted c\`adl\`ag jump process $(\boldsymbol{\tilde{J}}^\alpha_t)$ with state space $\S=\{0,1\}^{n(\alpha)}$. We then define the $\R^{n(\alpha)\times n(\alpha)}$-valued process $(\boldsymbol{J}^\alpha_t)$ by
$J_t^{\alpha,(i,j)} = \delta_{ij}\boldsymbol{\tilde{J}}^\alpha_t$.
As before, $(\boldsymbol{J}^\alpha_t)$ is a diagonal matrix-valued process indicating which coordinates of $(\boldsymbol{\xi}^\alpha_t)$ are active through $(\boldsymbol{J_t}^\alpha\,\boldsymbol{\xi}^\alpha_t)$ for $\alpha=1,\ldots,m$. 
We define a sub-algebra $\H_t\subseteq \mathcal{G}_t$ by
\begin{align}\label{subFextended} 
\H_t &= \s\left( (\boldsymbol{\tilde{J}}^\alpha_u \otimes \boldsymbol{\xi}^\alpha_u)_{0\leq u \leq t}, (\boldsymbol{\tilde{J}}^\alpha_u)_{0\leq u \leq t}, \boldsymbol{X}\1_{\{t=1\}}; \alpha=1,\ldots,m \right), \nonumber
\end{align}
for $t\in\mathbb{T}$, or equally $\H_t = \s\left( (\boldsymbol{J}^\alpha_u\,\boldsymbol{\xi}^\alpha_u)_{0\leq u \leq t}, (\boldsymbol{J}^\alpha_u)_{0\leq u \leq t}, \boldsymbol{X}\1_{\{t=1\}}; \alpha=1,\ldots,m \right)$. 

We define $\R$-valued effective information processes $(\hat{\xi^\alpha_t})_{t\in\mathbb{T}}$ for $\alpha=1,\ldots,m$ by
\begin{align*}
\hat{\xi^\alpha_t}=(\hat{\s^\alpha_{t}})^2 \sum_{i \in \mathcal{J}^\alpha_t} (\s^\alpha_i)^{-2}\xi^{\alpha,(i)}_t, \hspace{0.1in} \text{where} \hspace{0.1in} \hat{\s^\alpha_{t}} = \left(\sum_{i \in \mathcal{J}^\alpha_t} (\s^\alpha_i)^{-2}\right)^{-1/2},
\end{align*}
where $\mathcal{J}^\alpha_t = \{i : J_t^{\alpha,(i,i)} = 1\}$. We set $\hat{\xi^\alpha_t}=0$ and $\hat{\s^\alpha_{t}}=0$ when $\mathcal{J}^\alpha_t=\emptyset$. In a similar fashion, we define function-valued complementary information processes $(\eta^\alpha_t)_{t\in\mathbb{T}}$ for $\alpha=1,\ldots,m$ by
\begin{equation*}
\eta^\alpha_t: x^\alpha\mapsto \prod_{i\in\mathcal{J}_t^{\alpha \complement}} h\left(x^\alpha, \xi^{\alpha,(i)}(\t^{\alpha,(i)}_t), \t^{\alpha,(i)}_t, \s_i^\alpha\right),
\end{equation*}
where $\mathcal{J}_t^{\alpha \complement} = \{i : J_t^{\alpha,(i,i)} = 0\}$ and $\t^{\alpha,(i)}_t = 0 \vee \sup\{ u : J_u^{\alpha,(i,i)} = 1, u\in[0,t]\}$ for $i=1,\ldots,n(\alpha)$, where by convention $\sup\emptyset = -\infty$. We set $\eta^\alpha_t=1$ when $\mathcal{J}_t^{\alpha \complement} = \emptyset$.
\begin{prop} \label{effectivecomplementaryextended}
The measure $\nu_t(\dd \boldsymbol{x})=\P[\boldsymbol{X} \in \dd \boldsymbol{x} \,|\, \H_t]$ satisfies
\begin{align*}
\nu_t(\dd \boldsymbol{x})	= \prod_{\alpha=1}^m \P[X^\alpha \in \dd x^\alpha \,|\, \hat{\xi}^\alpha_t, \eta^\alpha_t], 
\end{align*}
for  $t \in \mathbb{T}_{*}$, where $(\hat{\xi}^\alpha_t)$ and $(\eta^\alpha_t)$ are defined as above for $\alpha=1,\ldots,m$.
\end{prop}
\begin{proof}
Using the independence properties given above, we have
\begin{align*}
\nu_t(\dd \boldsymbol{x}) &= \prod_{\alpha=1}^m \P[X^\alpha \in \dd x^\alpha \,|\, \s\left( (\boldsymbol{J}^\alpha_u\,\boldsymbol{\xi}^\alpha_u)_{0\leq u \leq t}, (\boldsymbol{J}^\alpha_u)_{0\leq u \leq t}, X^\alpha\1_{\{t=1\}}\right)] \nonumber \\
&= \prod_{\alpha=1}^m \frac{\prod_{i\in\mathcal{J}_t^\alpha} h\left(x^\alpha, \xi^{\alpha,(i)}(\t^{(i)}_t), \t^{\alpha, (i)}_t, \s_i^\alpha\right)
									\prod_{i\in\mathcal{J}_t^{\complement,\alpha}} h\left(x^\alpha, \xi^{\alpha,(i)}(\t^{\alpha, (i)}_t), \t^{\alpha, (i)}_t, \s_i^\alpha\right) \nu(\dd x^\alpha)}
									{\int_{\mathbb{X}} \prod_{i\in\mathcal{J}_t^\alpha} h\left(x^\alpha, \xi^{\alpha,(i)}(\t^{\alpha,(i)}_t), \t^{\alpha,(i)}_t, \s_i^\alpha\right)
											\prod_{i\in\mathcal{J}_t^\complement,\alpha} h\left(x^\alpha, \xi^{\alpha,(i)}(\t^{\alpha,(i)}_t), \t^{\alpha,(i)}_t, \s_i^\alpha\right) \nu(\dd x^\alpha)} \nonumber \\
&= \prod_{\alpha=1}^m\frac{h\left(x^\alpha, \hat{\xi}^\alpha_t, t, \hat{\s_{t}^\alpha} \right) \eta^\alpha_t(x^\alpha) \, \nu^\alpha(\dd x^\alpha)}
									{\int_{\mathbb{X}} h\left(x^\alpha, \hat{\xi}^\alpha_t, t, \hat{\s_{t}^\alpha} \right) \eta^\alpha_t(x^\alpha) \,\nu^\alpha(\dd x^\alpha)},
\end{align*}
by following similar mid-steps as done in Proposition \ref{effectivecomplementary}.
\end{proof}
We define $\pi^\alpha_j = \inf\{t:C^\alpha_t=j\}$,
with $\pi^\alpha_0=0$, where $C^\alpha_t = \sum_{s\leq t} \1\{\boldsymbol{\tilde{J}}^\alpha_{s}\ne\boldsymbol{\tilde{J}}^\alpha_{s-}\}$, and also $N^\alpha_t = \sum_{s\leq t} \delta^\alpha_s$, where $\delta^\alpha_s = \1\{\mathcal{J}^\alpha_s \backslash \mathcal{J}^\alpha_{s-} \ne \emptyset \}$ for $\alpha=1,\ldots,m$.
\begin{prop} \label{ExpectationSDEextended}
Let $(W_t^{\alpha,(j)})$ be mutually independent $(\H_t)$-Brownian motions across $\alpha=1,\ldots,m$ for $t\in[\pi^\alpha_{j-1}, \pi^\alpha_j \wedge 1)$. The $(\H_t)$-martingale $(X_t)_{t\in\mathbb{T}}$ defined by $X_t:=\mathbb{E}[X \,|\, \H_t]$, satisfies
\begin{align*}
	X_t = \E[X] &+ \sum_{\alpha=1}^{m}\sum_{j=1}^{C^\alpha_t}\left(\int_{\pi^\alpha_{j-1}}^{\pi^\alpha_j} \frac{\Theta^\alpha_s}{\hat{\s_{s}^\alpha} (1-s)}  \d W_s^{\alpha,(j)}\right)\1\{\mathcal{J}^\alpha_{\pi^\alpha_{j-1}}\neq\emptyset \} \nonumber \\
	&+\sum_{\alpha=1}^{m}\left(\int_{\pi_{C^\alpha_t}}^{t} \frac{\Theta^\alpha_s}{\hat{\s_{s}^\alpha} (1-s)} \d W_s^{\alpha(C_t+1)}\right)\1\{\mathcal{J}^\alpha_{\pi^\alpha_{C_t}}\neq\emptyset \} + \sum_{\alpha=1}^{m}\sum_{s\leq t} (X_s - X_{s-})\delta^\alpha_s,
\end{align*}
for $t\in\mathbb{T}_{*}$, where $(\Theta^\alpha_t)_{t\in\mathbb{T}}$ is given by $\Theta^\alpha_t =\cov[X,X^\alpha\,|\,\H_t]$ for $\alpha=1,\ldots,m$.
\end{prop}
\begin{proof}
Let $\nu_t^\alpha(\dd x^\alpha)=\P[X^\alpha \in \dd x^\alpha \,|\, \hat{\xi}^\alpha_t, \eta^\alpha_t]$ for $\alpha=1,\ldots,m$. Then using the independence properties given above, Proposition \ref{effectivecomplementaryextended} and following similar steps as for Proposition \ref{measureSDE}, there exists a set of $(\H_t)$-Brownian motions $(W_t^{\alpha,(j)})$ over random periods $t\in[\pi^\alpha_{j-1}, \pi^\alpha_j \wedge 1)$, which are mutually independent from each other across $\alpha=1,\ldots,m$, where the brackets $\langle \d W_t^{a,(j)},\d W_t^{b,(j)}\rangle = 0$ for $a,b\in\{1,\ldots,m\}$, such that $(\nu_t)$ satisfies
\begin{align*}
\nu_t(\dd \boldsymbol{x}) &\law \nu_0(\dd \boldsymbol{x}) + \sum_{\alpha=1}^{m}\sum_{j=1}^{C_t^\alpha}\left(\int_{\pi^\alpha_{j-1}}^{\pi^\alpha_j} \frac{x^\alpha-\E[X^\alpha\,|\, \hat{\xi}^\alpha_s, \eta^\alpha_s]}{\hat{\s_{s}^\alpha} (1-s)} \prod_{\alpha=1}^m\nu^\alpha_s(\dd x^\alpha) \d W_s^{\alpha,(j)}\right)\1\{\mathcal{J}^\alpha_{\pi^\alpha_{j-1}}\neq\emptyset \} \nonumber \\ 
&+\sum_{\alpha=1}^{m}\left(\int_{\pi_{C^\alpha_t}}^{t}\frac{x^\alpha-\E[X^\alpha\,|\, \hat{\xi}^\alpha_s, \eta^\alpha_s]}{\hat{\s_{s}^\alpha} (1-s)} \prod_{\alpha=1}^m\nu^\alpha_s(\dd x^\alpha) \d W_s^{\alpha,(C_t+1)}\right)\1\{\mathcal{J}^\alpha_{\pi^\alpha_{C_t}}\neq\emptyset \} \nonumber \\ 
&+ \sum_{\alpha=1}^{m}\sum_{s\leq t} (\nu_s(\dd \boldsymbol{x}) - \nu_{s-}(\dd \boldsymbol{x}))\delta_s^\alpha,
\end{align*}
for $t \in \mathbb{T}_{*}$. From Proposition \ref{effectivecomplementaryextended}, we have $\nu_t(\dd \boldsymbol{x}) = \prod_{\alpha=1}^m\nu^\alpha_t(\dd x^\alpha)$. 
Therefore, since $g:\mathbb{X}^m\rightarrow\mathbb{X}$ is a bounded measurable function, we have
\begin{align*} 
X_t = \int_{\mathbb{X}^m}g(x^1,\ldots,x^m)\nu^1_t(\dd x^\alpha)\ldots\nu^m_t(\dd x^\alpha), 
\end{align*}
and the statement follows via Lebesgue integration.
\end{proof}
The diffusion coefficient of $(X_t)$ takes the form of a stochastic covariance process, which itself jumps when any of the information sources switches on. Hence, $(X_t)$ adapts to a new covariance state of the system under a new set of active information configuration.

The next statement generalizes Proposition \ref{feynmankac} for the functional $\E\left[\left. . \,\right| \H_t\right]$. As before, for $t<1$, we set $\mu^\alpha(\hat{\xi}^\alpha_t, \eta^\alpha_t,t) =(\E[ X^\alpha \,|\, \hat{\xi}^\alpha_t, \eta^\alpha_t ] - \hat{\xi}^\alpha_t)/(1-t)$ when $\mathcal{J}^\alpha_{t}\neq\emptyset$ and $\mu^\alpha(\hat{\xi}^\alpha_t, \eta^\alpha_t,t) =0$ otherwise, for $\alpha=1,\ldots,m$.
\begin{prop}\label{feynmankacextended}
Let $\boldsymbol{\hat{\xi}}_t = [\hat{\xi}^1_t,\ldots,\hat{\xi}^m_t]$, $\boldsymbol{\eta}_t = [\eta^1_t,\ldots,\eta^m_t]$, and $(\lambda^\alpha_t)$ be the intensity process of $(N^\alpha_t)$. Let $\psi:\R\rightarrow\R$ and $\phi:\R\rightarrow\R$ be continuous bounded functions. Then the conditional expectation
	\[v(\boldsymbol{\hat{\xi}}_t, \boldsymbol{\eta}_t,t) = \E\left[\left. e^{-\int_t^1 \psi(s)\dd s}\phi(X) \,\right| \boldsymbol{\hat{\xi}}_t, \boldsymbol{\eta}_t,\boldsymbol{X}\1_{\{t=1\}}\right], \]
	satisfies the partial differential equation
\begin{multline*}
		\frac{\partial v(\boldsymbol{\hat{\xi}}_t, \boldsymbol{\eta}_t,t)}{\partial t} 
			+ \sum_{\alpha=1}^m\mu^\alpha(\hat{\xi}_t, \eta_t,t)\frac{\partial v(\boldsymbol{\hat{\xi}}_t, \boldsymbol{\eta}_t,t)}{\partial \hat{\xi}^\alpha_t} 
			+ \sum_{\alpha=1}^m(\hat{\s}^\alpha_t)^2\frac{\partial^2 v(\boldsymbol{\hat{\xi}}_t, \boldsymbol{\eta}_t,t)}{2(\partial \hat{\xi}^\alpha_t)^2}
		\\ 
			- \psi(t)v(\boldsymbol{\hat{\xi}}_t, \boldsymbol{\eta}_t,t) 
			+ \sum_{\alpha=1}^m(v(\boldsymbol{\hat{\xi}}_t, \boldsymbol{\eta}_t,t) 
			- v(\boldsymbol{\hat{\xi}}_{t-}, \boldsymbol{\eta}_{t-},t-))\lambda^\alpha_t = 0,
\end{multline*}
with the boundary condition $v(\boldsymbol{\hat{\xi}}_1, \boldsymbol{\eta}_1,1) = \phi(X)$.
\end{prop}
We omit the proof of Proposition \ref{feynmankacextended} to avoid repetition, which follows from similar steps as done in Proposition \ref{measureSDE} and Proposition \ref{feynmankac} by overlaying the independence properties given above, where we see that the brackets $\langle \d\hat{\xi}_t^{a,(j)},\d \hat{\xi}_t^{b,(j)}\rangle = 0$ for $a\neq b$ and $a,b\in\{1,\ldots,m\}$.

\subsection{Modulation as projection}
\label{modprojsec}
For any $t\in\mathbb{T}$ and for any fixed $\alpha=1,\ldots,m$, $\boldsymbol{\tilde{J}}^\alpha_t \otimes \,\boldsymbol{\xi}^\alpha_t \mapsto\boldsymbol{J}^\alpha_t\,\boldsymbol{\xi}^\alpha_t$ defines an orthogonal projection from the space of all available information sources to an information subspace. That is, $(\boldsymbol{J}^\alpha_t)$ is a projection-valued stochastic process acting on $\boldsymbol{\xi}^\alpha_t$, whereas $\boldsymbol{A}^\alpha_t = \boldsymbol{I} - \boldsymbol{J}^\alpha_t$, with $\boldsymbol{I}$ the identity matrix, is an annihilator matrix determining a complementary information subspace. This projection can also be formalized by starting from an enlarged filtration, where 
\begin{equation}\label{enlargedfiltration} 
\Z_t = \s\left( (\boldsymbol{\xi}^\alpha_u)_{0\leq u \leq t}, (\boldsymbol{J}^\alpha_u)_{0\leq u \leq t}, \boldsymbol{X}\1_{\{t=1\}}; \alpha=1,\ldots,m \right), \nonumber
\end{equation}
is a subalgebra $\Z_t \subseteq \mathcal{G}_t$ for $t\in\mathbb{T}$. Thus, $\H_t \subseteq \Z_t$ and 
\begin{align}
\P[\boldsymbol{X} \in \dd \boldsymbol{x} \,|\, \H_t] = \E\left[\P\left[\boldsymbol{X} \in \dd \boldsymbol{x}) \,|\, \Z_t \right] \,|\, \H_t \right],  \notag
\end{align}
is a projection. Having $\xi^{\alpha,(i)}_t \law t X^\alpha + \s_i^\alpha \b_t^{\alpha,(i)}$,
for $\s_i^\alpha>0$ and $i=1,\ldots,n(\alpha)$, we can outright define the full effective processes
\begin{align*}
\Psi^\alpha_t=(\hat{\epsilon^\alpha})^2 \sum_{i =1}^{n(\alpha)} (\s^\alpha_i)^{-2}\xi^{\alpha,(i)}_t, \hspace{0.1in} \text{where} \hspace{0.1in} \hat{\epsilon^\alpha} = \left(\sum_{i=1}^{n(\alpha)} (\s^\alpha_i)^{-2}\right)^{-1/2},
\end{align*}
without the need for any complementary process, and where $\hat{\epsilon^\alpha}$s are constants and $\hat{\epsilon^\alpha}>0$ for $\alpha=1\ldots,m$. Accordingly, we can define $(\Z_t)$-martingales $(M_t^\alpha)_{t\in\mathbb{T}_*}$ by
\begin{align*}
M_t^\alpha = \Psi_t^\alpha - \int_0^t \frac{\E[ X^\alpha \,|\, \Psi^\alpha_t ] - \Psi_t^\alpha}{ 1 - s}\dd s,
\end{align*}
where the martingale property follows as a simple case of Proposition \ref{piecewisemartingale}.
Therefore, $(B_t^\alpha)_{t\in\mathbb{T}_*}$ defined by $B_t^\alpha = M_t^\alpha / \epsilon^\alpha$ are $(\Z_t)$-Brownian motions for $\alpha=1\ldots,m$. Hence, using Proposition \ref{effectivecomplementary} and following similar steps as done in Proposition \ref{ExpectationSDEextended}, the measure $\hat{\nu}_t(\dd \boldsymbol{x})=\P[\boldsymbol{X} \in \dd \boldsymbol{x} \,|\, \Z_t]$ satisfies
\begin{align*}
\hat{\nu}_t(\dd \boldsymbol{x}) &= \P\left[\boldsymbol{X} \in \dd \boldsymbol{x}) \,| \{ \Psi^\alpha_t : \alpha=1\ldots,m \}\right] \notag \\
&=\hat{\nu}_0(\dd \boldsymbol{x}) + \sum_{\alpha=1}^{m}\int_{0}^{t} \frac{x^\alpha-\E[X^\alpha\,|\, \Psi^\alpha_s]}{\epsilon^\alpha(1-s)} \hat{\nu}_s(\dd \boldsymbol{x}) \d B_s^{\alpha} \nonumber \\
&=\hat{\nu}_0(\dd \boldsymbol{x}) + \int_{0}^{t}\hat{\nu}_s(\dd \boldsymbol{x})\sum_{\alpha=1}^{m} \Gamma_s^\alpha \d B_s^{\alpha}, \nonumber
\end{align*}
where $(\Gamma_s^\alpha)_{0\leq t < 1}$ is defined correspondingly as above for $\alpha=1\ldots,m$. Then, following similar steps as done in Proposition \ref{effectivecomplementaryextended} and solving the SDE above, we have
\begin{align*}
\nu_t(\dd \boldsymbol{x})	&= \E\left[ \hat{\nu}_t(\dd \boldsymbol{x}) \,\left.|\, \hat{\xi}^\alpha_t, \eta^\alpha_t \right.\right] \notag \\
&= \E\left[\prod_{\alpha=1}^m \P\left[ X^\alpha \in \dd x^\alpha \,\bigg|\, \{ \Psi^\alpha_t : \alpha=1\ldots,m \} \right]  \,\bigg|\, \hat{\xi}^\alpha_t, \eta^\alpha_t \right] \notag \\
&= \nu_0(\dd \boldsymbol{x})\E\left[\exp\left(-\frac{1}{2} \int_0^t \sum_{\alpha=1}^{m} (\Gamma_s^\alpha)^2  \d s + \int_0^t \sum_{\alpha=1}^{m} \Gamma_s^\alpha \d B_s^{\alpha} \right) \,\bigg|\, \hat{\xi}^\alpha_t, \eta^\alpha_t \right],
\end{align*}
since $\nu_0(\dd \boldsymbol{x}) = \hat{\nu}_0(\dd \boldsymbol{x})$, the bracket $\langle \d B_t^{\alpha},\d B_t^{\beta}\rangle = 0$ for $\alpha,\beta\in\{1,\ldots,m\}$, and since Novikov's condition holds:
\begin{align*}
\E\left[\exp\left(\frac{1}{2} \int_0^t \bigg| \sum_{\alpha=1}^{m}\sum_{\beta=1}^{m}  \Gamma_s^\alpha \Gamma_s^\beta \bigg|^2 \d \langle B_s^{\alpha}, B_s^{\beta} \rangle \right) \right] < \infty.
\end{align*}
This motivates us to ask how far we can push this idea towards its logical limits for modelling information modulation. To maintain generality, we let $(\boldsymbol{\xi}^\alpha_t)_{t\in\mathbb{T}}$ below to be LRBs, not necessarily Brownian bridge information processes, nor even the same class of LRB across $\alpha=1\ldots m$.
\begin{defn}
A \emph{modulated \levy random bridge information system} is given by the probability space $(\Omega, \mathcal{H},\P)$ equipped with a right-continuous and complete filtration $(\mathcal{H}_t)_{t\in\mathbb{T}}$, where
\begin{enumerate}
\item $\H_t = \s\left( (\boldsymbol{P}^\alpha_u\,\boldsymbol{\xi}^\alpha_u)_{0\leq u \leq t}, (\boldsymbol{P}^\alpha_u)_{0\leq u \leq t}, \boldsymbol{X}\1_{\{t=T\}} ; \alpha=1,\ldots,m \right)$.
\item $(\boldsymbol{\xi}^\alpha_t)_{t\in\mathbb{T}}$ is an $\R^{n(\alpha)}$-valued \levy random bridge with generating law $\nu^{\alpha}$ for $\alpha=1,\ldots,m$.
\item $(\boldsymbol{P}^\alpha_t)_{t\in\mathbb{T}}$ is an $n(\alpha)$-dimensional projection-valued stochastic process with real c\`adl\`ag paths for $\alpha=1,\ldots,m$.
\item $\boldsymbol{X}\in\mathcal{L}^2(\Omega,\mathcal{G},\P)$ is an $m$-dimensional vector of random variables with law $\boldsymbol{\nu}$ and state-space $(\mathbb{X}^m,\otimes^{m}_{\alpha=1}\mathcal{B}(\mathbb{X}))$, where $\mathbb{X}\subseteq\R$ and $\mathcal{B}(\mathbb{X})$ is the Borel $\sigma$-field.
\item $\mathbb{T}=[0,T]$ is a finite interval for some $T<\infty$.
\end{enumerate}
\end{defn}
This definition of the system does not explicitly rely on any random point field. However, one can assign a law to $(\boldsymbol{P}^\alpha_t)_{t\in\mathbb{T}}$ via associating it with a random point field $\zeta^\alpha$. Although we leave a detailed study of such a system for future research, we shall nonetheless provide a simple example to show how---what we call---\emph{information mixing} arises.
\begin{ex}
Let $m=1$, $(\boldsymbol{\xi}_t)_{t\in\mathbb{T}}$ is such that $\xi^{(i)}_t \law t X + \s_i \b_t^{(i)}$ for $\s_i^\alpha>0$, and $p_t^{(ij)}$ denote the $(i,j)$th coordinate of an $n$-dimensional projection matrix $\boldsymbol{P}_t$. Then at any $t\in\mathbb{T}$, the $\sigma$-algebra $\H_t$ allows the mapping $\boldsymbol{P}_t\,\boldsymbol{\xi}_t \mapsto \boldsymbol{\psi}_t$, where
\[
	\psi_t^{(i)} = 
	\left\{ \begin{aligned}
		&0, && \text{if $p_t^{(ij)} =0$ for all $j\in\{1,\ldots,n\}$,}
		\\ &tX + \hat{p}_t^{(i)}\sum_{j=1}^n p_t^{(ij)}\sigma_j\beta_t^{(j)}, &&\text{otherwise,}
	\end{aligned} \right.
\]
and where $\hat{p}_t^{(i)} = \left(\sum_{j=1}^n p_t^{(ij)}\right)^{-1}$. Thus, $\psi_t^{(i)}\law tX + \alpha^{(i)}_tB^{(i)}_t$ given that $(B^{(i)}_t)_{t\in\mathbb{T}}$ is a standard Brownian bridge and the bridge coefficient is
\begin{align*}
\alpha^{(i)}_t = \hat{p}_t^{(i)}\left(\sum_{j=1}^n (p_t^{(ij)})^2\sigma_j^2\right)^{\frac{1}{2}}.
\end{align*}   
When $\boldsymbol{P}_t$ is diagonal (e.g. $\boldsymbol{P}_t = \boldsymbol{J}_t$), the observer of $\boldsymbol{P}_t\,\boldsymbol{\xi}_t$ also observes the active coordinates of $\boldsymbol{\xi}_t$, possibly scaled. Generally, the observer may not be able to decouple all the active coordinates of $\boldsymbol{\xi}_t$ from the observed $\boldsymbol{P}_t\,\boldsymbol{\xi}_t$, since $\boldsymbol{P}_t$ is non-invertible, unless $\boldsymbol{P}_t=\boldsymbol{I}$, for it is a projection. 
Hence, the original information provided from each source may get mixed with another. 
\end{ex}

\section{Applications} \label{sec:application}
Thus far, we have not specified an interpretation for $X$. Next, we apply the framework to finance in the spirit of \cite{4}. We assume that $\P$ is the pricing measure, $(\F_t)$ is the market information and $X$ is the cash flow of a financial asset at $t=1$. We let the risk-free system be deterministic and denote the system of discount functions by $(P_{0t})_{0 \leq t <\infty}$. We assume that $P_{0t}$ is differentiable, strictly decreasing and that it satisfies $0 < P_{0t} \leq 1$ and $\lim_{t\rightarrow\infty}P_{0t}=0$. The no-arbitrage condition implies that $P_{t1}=P_{01}/P_{0t}$ for $t\leq 1$. 
We write $\overline{X}_t = P_{t1}X_t = P_{t1}\E[X\,|\,\F_t]$ for the price at time $t$. The price process $(\overline{X}_t)$ has jump-diffusion dynamics, see Section \ref{EJD}.    

\subsection{Merton-type jump-diffusion models for vanilla options}
We choose $\zeta$ such that $\tilde{J}_t^{(i)}=1$ implies $\tilde{J}_u^{(i)}=1$ for $u\geq t$. This removes any path-dependency in the system caused by randomly paused flows of information. We also let there be at least one active source of information at $t=0$, so that some information about $X$ is always available to market participants. In this setting, if $\xi^{(i)}_t \law t X + \s_i \b_t^{(i)}$, Proposition \ref{ExpectationSDEinit} reduces down to 
\begin{align*}
	X_t &= \E[X] + \sum_{j=1}^{C_t}\int_{\pi_{j-1}}^{\pi_j} \theta_s  \d W_s^{(j)} +\int_{\pi_{C_t}}^{t} \theta_s \d W_s^{(C_t+1)} + \sum_{s\leq t} (X_s - X_{s-})\delta_s.
\end{align*}
As an example, we shall price a European-style call option with strike price $K$ and exercisable at a fixed time $t\in(0,1)$. The price at time zero is given by $C_{0}=P_{0t}\E\left[(\overline{X}_{t}-K)^{+}\right]$ where $0 < t < 1$. , and to make the derivation easier to follow, we represent processes explicitly as a function of the state $k\in \S$, e.g., we write $\hat{\xi}_t(k)$ for $\boldsymbol{\tilde{J}}_t=k$.

\begin{prop} \label{callprice}
The price at $t=0$ of the European-style call option is given by
\begin{align}
C_{0}&=P_{0t}\sum_{k\in\S}\P(\boldsymbol{\tilde{J}}_{t} = k)\int_{\mathbb{X}}x\mathcal{N}\left(-\hat{z}_{t}(k)+\frac{x\sqrt{t}}{\hat{\sigma}_t^2(k)\sqrt{(1-t)}}\right)\nu(\dd x) \nonumber \\
&-P_{0t}\sum_{k\in\S}\P(\boldsymbol{\tilde{J}}_{t} = k)K\int_{\mathbb{X}}\mathcal{N}\left(-\hat{z}_{t}(k)+\frac{x\sqrt{t}}{\hat{\sigma}_t^2(k)\sqrt{(1-t)}}\right)\nu(\dd x),\nonumber
\end{align}
for $t\in(0,1)$, where $\mathcal{N}(\cdot)$ is the standard normal distribution and $\hat{z}_{t}(k)=\hat{\varsigma}(k)/\sqrt{t(1-t)}$, where $\hat{\varsigma}(k)$ is the unique solution to
\begin{align}
\int_{\mathbb{X}}(P_{t1}x-K)\exp\left[\frac{1}{\hat{\sigma}_t^2(k)(1-t)}\left(x \hat{\varsigma}(k)-tx^2/2\right)\right]\nu(\dd x)=0. \nonumber
\end{align}
\end{prop}
\begin{proof}
Let $\overline{X}_t(k) = P_{t1}\E[X\,|\,\hat{\xi}_t(k)]$,
which is the price of the asset given that $\boldsymbol{\tilde{J}}_t$ is in state $k\in\S$. That is, since $\hat{\xi}_t$ has to be in one of the $2^n$ states at any time, $k$ should be understood as an identifier of the active and inactive coordinates of $\boldsymbol{\xi}_t$. 
Then, the price of the option is given by the following:
\begin{align}
C_{0}=\sum_{k\in\S}C_0(k)\P(\boldsymbol{\tilde{J}}_t=k), \nonumber
\end{align}
for $t\in(0,1)$, where $C_0(k)=P_{0t}\E\left[(\overline{X}_t(k)-K)^{+}\right]$.
This follows from the law of total expectation, where we have
\begin{align}
\E\left[(\overline{X}_{t}-K)^{+}\right]&=\E\left[\E\left[\left.(\overline{X}_{t}-K)^{+} \,\right|\, \boldsymbol{\tilde{J}}_{t}\right]\right] \notag \\
&=\sum_{k\in\S}\E\left[\left.(\overline{X}_{t}-K)^{+} \,\right|\, \boldsymbol{\tilde{J}}_{t}=k\right]\P(\boldsymbol{\tilde{J}}_{t} = k) \notag \\
&=\sum_{k\in\S}\E\left[(\overline{X}_{t}(k)-K)^{+}\right]\P(\boldsymbol{\tilde{J}}_{t}= k), \nonumber
\end{align}
since $(\boldsymbol{\xi}_t)$ is Markovian and independent of the state process $(\boldsymbol{\tilde{J}}_t)$, and $\eta_t=1$ for all $t\in\mathbb{T}$.
We can write the conditional distribution as an explicit function of $k\in\S$,
\begin{align}
\nu_{t}(\dd x;k)=\frac{\exp\left[\frac{1}{(1-t)}\hat{\sigma}_t^{-2}(k)\left(x\hat{\xi}_{t}(k)-tx^2/2\right)\right]\nu(\dd x)}{\int_{\mathbb{X}}\exp\left[\frac{1}{(1-t)}\hat{\sigma}_t^{-2}(k)\left(x\hat{\xi}_{t}(k)-tx^2/2\right)\right]\nu(\dd x)}, \nonumber
\end{align}
for $t\in(0,1)$. Following similar steps as in \cite{4} while decomposing with respect to the state $k\in\S$, we have
\begin{align}
C_{0}(k)=P_{0t}\E\left[\frac{1}{\Phi_{t}(k)}\left(\int_{\mathbb{X}}(P_{t1}x-K)\chi_{t}(x;k)\nu(\dd x)\right)^{+}\right], \nonumber
\end{align}
where $\Phi_{t}(k)=\int_{\mathbb{X}}\chi_{t}(x;k)\nu(\dd x)$ and
\[
\chi_{t}(x;k)=\exp\left(\frac{x \hat{\xi}_{t}(k)-tx^2/2}{\hat{\sigma}_t^{2}(k)(1-t)}\right).
\] 
Then, for $t\in(0,1)$, we can define the measure $\mathbb{B}$ on $(\Omega,\mathcal{G},\{\mathcal{G}_{t}\})$ through a sequence of Radon-Nikodym derivatives
\begin{equation}
\left\{\left.\frac{\d \mathbb{B}}{\d \mathbb{P}}\right | _{{\sigma((\hat{\xi}_{s}(k))_{0\leq s \leq t})}}\right\}_{k\in\S}=\left\{\frac{1}{\Phi_{t}(k)}\right\}_{k\in\S}, \nonumber
\end{equation}
This follows because the process $(1/\Phi_{t}(k))$ is a $\P$-martingale; that is,
\begin{align}
\E\left[1/\Phi_{t}(k)|\hat{\xi}_{s}(k)\right]=1/\Phi_{s}(k)
\end{align}
for $s<t$, and also $\Phi_{0}(k)=1$ and $\Phi_{t}(k)>0$. 
In particular,
\begin{align}
(\Phi_{t}(k))^{-1}=\exp\left(-\int^{t}_{0}\frac{\E\left[X|\hat{\xi}_s(k)\right]}{\hat{\sigma}_s(k)(1-s)}\d W_{s}(k)-\frac{1}{2}\int^{t}_{0}\frac{\E\left[X|\hat{\xi}_{s}(k)\right]^2}{\hat{\sigma}_s^2(k)(1-s)^{2}}\dd s \right), \nonumber
\end{align}
and the Novikov's condition
\begin{align}
\E\left[\exp\left(\frac{1}{2}\int^{t}_{0}\frac{\E\left[X|\hat{\xi}_{s}(k)\right]^2}{\hat{\sigma}^2_s(k)(1-s)^{2}}\dd s \right)\right]<\infty, \notag
\end{align}
is satisfied. Under the measure $\mathbb{B}$, the random variable $\hat{\xi}_{t}(k)$ is Gaussian and we have
\begin{align}
C_{0}&=P_{0t}\sum_{k\in\S}\P(\boldsymbol{\tilde{J}}_{t} = k)\E^{\mathbb{B}}\left[\left(\int_{\mathbb{X}}(P_{t1}x-K)\chi_{t}(x;k)\nu(\dd x)\right)^{+}\right]. \nonumber
\end{align}
The statement follows by computing the critical value $\hat{\varsigma}(k)$.  
\end{proof}

The option price can be represented as the weighted sum of Black-Scholes-Merton prices induced by the various combinations of active information processes.
The functional form of the price is very similar to that presented in \cite{19} where jumps are assumed to be driven by a Poisson process. In our framework, the jump-diffusion dynamics of the price process emerges from the nature of the market information (i.e. \textit{endogenous}), the jump distribution is not specified {\it a priori} and it corresponds to the much larger class of random counting measures. 

\begin{ex}
As a special case, we can choose the underlying random point field to generate an independent Poisson process with intensity $\lambda$ such that we have
\begin{align}
C_{0}&=P_{0t}\sum_{k=1}^\infty\frac{e^{-\lambda t}(\lambda t)^{k-1}}{(k-1)!}\int_{\mathbb{X}}x\mathcal{N}\left(-\hat{z}_{t}(k)+\frac{x\sqrt{t}}{\hat{\sigma}_t^2(k)\sqrt{(1-t)}}\right)\nu(\dd x) \nonumber \\
&-P_{0t}\sum_{k=1}^\infty\frac{e^{-\lambda t}(\lambda t)^{k-1}}{(k-1)!}K\int_{\mathbb{X}}\mathcal{N}\left(-\hat{z}_{t}(k)+\frac{x\sqrt{t}}{\hat{\sigma}_t^2(k)\sqrt{(1-t)}}\right)\nu(\dd x),\nonumber
\end{align}
\end{ex}

We refer to \cite{9} and \cite{10} for partial integro-differential equations and viscosity solutions that offer alternative techniques for option pricing with respect to jump-diffusion dynamics.
\subsection{Information asymmetry and market competition}
We consider a setting in which there are two market agents who are unaware of each other's actions and have differing access to a fixed universe of information sources.
We here allude to a randomly evolving competition between agents, where some may have an informational advantage over a period of time and suddenly find that the wheel has turned with them having to come to terms with several of their information sources being disabled. 

Let $(\boldsymbol{\tilde{J}}_t^{(1)})\in\S$ and $(\boldsymbol{\tilde{J}}_t^{(2)})\in\S$, which are $\mathcal{G}_t$-adapted c\`adl\`ag jump processes of Agents 1 and 2, respectively. The jump processes are generated by two random point fields $\zeta^{(1)}$ and $\zeta^{(2)}$, both independent of $(\boldsymbol{\xi}_t)$.
We define the sub-algebras $\F_t^{(j)}\subseteq \mathcal{G}_t$ by
\[ \F_t^{(j)} = \s\left( (\boldsymbol{\tilde{J}}_u^{(j)} \otimes \boldsymbol{\xi}_u)_{0\leq u \leq t}, (\boldsymbol{\tilde{J}}_u^{(j)})_{0\leq u \leq t}, X1_{\{t=T\}}  \right), \]
for $j=\{1,2\}$.
Then we introduce $\nu_t^{(j)}(A) = \P[X \in A \,|\, \F_t^{(j)}]$,
where, if we use Proposition \ref{effectivecomplementary}, we have $\nu_t^{(j)}(A) = \P[X \in A \,|\, \hat{\xi}_t^{(j)}, \eta_t^{(j)}]$.
We note that $\P[\hat{\xi}_t^{(1)}=\hat{\xi}_t^{(2)}]<1$ and $\P[\eta_t^{(1)}=\eta_t^{(2)}]<1$, since the set-valued processes $(\mathcal{J}_t^{(j)})$, which define the effective and complementary information processes of the agents, are different. To model the information asymmetry in a dynamic competition between the agents, we shall use $f$-divergences; see \cite{1a, 3, 11}. We can define a symmetric $f$-divergence between equivalent probability measures $\P^{(1)}$ and $\P^{(2)}$ as
\begin{equation}
\Delta_{f}\left[\P^{(1)}||\,\P^{(2)}\right]=\frac{1}{2}\left[\int_{\Omega}f\left(\frac{\dd \P^{(1)}(\omega)}{\dd \P^{(2)}(\omega)}\right)\dd \P^{(2)}(\omega) + \int_{\Omega}f\left(\frac{\dd \P^{(2)}(\omega)}{\dd \P^{(1)}(\omega)}\right)\dd \P^{(1)}(\omega)\right],\nonumber
\end{equation}
for $\omega\in\Omega$, where $f$ is a convex function that satisfies $f(1)=0$, and where $\dd \P^{(1)}/\dd \P^{(2)}$ is the Radon-Nikodym derivative. Alternatively, a symmetric $f$-divergence can be defined in terms of probability densities, assuming that they exist. For convenience, we assume that $X$ has a density and write $\nu_t^{(j)}(\dd x) = p_t^{(j)}(x)\dd x$ such that $p_t^{(j)}(x)>0$ for $t\in\mathbb{T}_{*}$ and $x\in\mathbb{X}$.

As an example, we shall utilise Kullback-Leibler divergence (relative entropy), which is not a distance metric on the space of probability distributions, but measures the information gain when moving from a prior distribution to a posterior distribution. 
For $t\in\mathbb{T}_{*}$,
\begin{align}
\Delta_{\text{KL}}\left[p_t^{(1)} ||\, p_t^{(2)}\right]&=\frac{1}{2}\int_{\mathbb{X}}\left[p_t^{(1)}(x)\log\left(\frac{p_t^{(1)}(x)}{p_t^{(2)}(x)}\right)+p_t^{(2)}(x)\log\left(\frac{p_t^{(2)}(x)}{p_t^{(1)}(x)}\right)\right]\dd x.\nonumber
\end{align}
We represent the probability density functions explicitly as a function of the state; $p_t^{(j)}(x;k_j)$ for $\boldsymbol{\tilde{J}}^{(j)}_t = k_j$,
where $k_j\in\S$ for $j=\{1,2\}$. Thus, if $\xi^{(i)}_t \law t X + \s_i \b_t^{(i)}$, we have $p_t^{(j)}(x;k_j)\dd x = \P[X \in \dd x \,|\, \hat{\xi}_t(k_j),\eta_t^{(j)}(k_j)]$.
Then, by defining the processes $(A_t(k_1,k_2))$ and $(B_t(k_1,k_2))$ by
\begin{align}
&A_t(k_1,k_2)=\int_{\mathbb{X}}\log\left(\frac{p_t^{(1)}(x;k_1)}{p_t^{(2)}(x;k_2)}\right)\nu_t^{(j)}(\dd x;k_1), 
&B_t(k_1,k_2)=\int_{\mathbb{X}}\log\left(\frac{p_t^{(2)}(x;k_2)}{p_t^{(1)}(x;k_1)}\right)\nu_t^{(2)}(\dd x;k_2), \nonumber
\end{align}
respectively, we may write 
\begin{align}
2\Delta_{\text{KL}}\left[p_t^{(1)} || p_t^{(2)}\right]&=\sum_{k_1\in\S}\sum_{k_2\in\S}\left(A_t(k_1,k_2) +
B_t(k_1,k_2) \right)\1\{\boldsymbol{\tilde{J}}_{t}^{(1)}= k_1\}\1\{\boldsymbol{\tilde{J}}_{t}^{(2)}= k_2\}.\nonumber
\end{align}
As a geometric measure for information asymmetry, since $\sqrt{p^{(j)}_{t}(x)}$ for $j=\{1,2\}$ determines a point on $\mathcal{S}^+\in\mathcal{L}^{2}$, we can define the angle process $\{\Theta^{1,2}_{t}\}_{0\leq t < T}$ by the $\mathcal{L}^{2}$-inner product,
\begin{align} \label{eq:bhatcossh}
\cos \Theta^{1,2}_{t}(k_1,k_2) &= \frac{\left\langle \sqrt{p^{(1)}_{t}(x;k_1)},\sqrt{p^{(2)}_{t}(x;k_2)} \right\rangle}{\left\langle \sqrt{p^{(1)}_{t}(x;k_1)},\sqrt{p^{(1)}_{t}(x;k_1)} \right\rangle^{\frac{1}{2}} \left\langle \sqrt{p^{(2)}_{t}(x;k_2)},\sqrt{p^{(2)}_{t}(x;k_2)} \right\rangle^{\frac{1}{2}}} \notag \\
&=\int_{\mathbb{X}}\sqrt{p^{(1)}_{t}(x;k_1)}\sqrt{p^{(2)}_{t}(x;k_2)}\d x\notag \\
&=1-\frac{1}{2}\int_{\mathbb{X}}\left(\sqrt{p^{(1)}_{t}(x;k_1)}-\sqrt{p^{(2)}_{t}(x;k_2)}\right)^{2}\d x, \notag
\end{align}
since the $\mathcal{L}^{2}$-norm of $\sqrt{p^{(j)}_{t}(x)}$ equals unity, and $\mathcal{S}\in\mathcal{L}^{2}$ is a Riemannian manifold when equipped with the inner product $\langle . \, ,. \rangle$. Here, $\Theta^{1,2}_{t}(k_1,k_2)=\Theta^{2,1}_{t}(k_2,k_1)$ is the Bhattacharyya angle between $p^{(1)}_t$ and $p^{(2)}_t$; the angle from the center of $\mathcal{S}$ subtended to the endpoints on $\mathcal{S}^{+}$.
\begin{align}
\cos \Theta^{1,2}_{t}&=\sum_{k_1\in\S}\sum_{k_2\in\S}\cos \Theta^{1,2}_{t}(k_1,k_2)\1\{\boldsymbol{\tilde{J}}_{t}^{(1)}= k_1\}\1\{\boldsymbol{\tilde{J}}_{t}^{(2)}= k_2\}.\nonumber
\end{align}
These are useful representations to derive SDEs for the Kullback-Leibler $\Delta_{\text{KL}}$ and spherical $\Theta$ asymmetry processes via Proposition \ref{measureSDE}.
These processes jump every time one of the agents gains access to a new source of information. If both agents have the same information at a given time, that is, if $\nu_t^{(1)}(\dd x;k_1)=\nu_t^{(2)}(\dd x;k_2)$ for some $t\in\mathbb{T}_{*}$, then the processes are zero at that time. The study of information asymmetry (see, for example \cite{1,2,18}) has benefited substantially from the theory of enlargements of filtrations; see \cite{17} and \cite{22}. 
\begin{rem}
The activation and deactivation of new information sources at random times can be understood as a dynamic interplay between enlargements of filtrations and stopped filtrations. 
\end{rem}
We leave a formal treatment of this remark and a corresponding rigorous application of information asymmetry to market competition with multiple agents for future research.

\subsection{Conclusions}
In order to model information switches, we have worked in the foregoing with filtrations generated by dynamically modulated multivariate \levy random bridges. In the case of Brownian random bridges, we have constructed what we call effective and complementary information processes, which provide us with analytical tractability to prove a dynamical representation for the conditional expectation martingale in terms of standard Brownian motions. The effective process is a one-dimensional aggregation of all active information sources and the complementary process reflects the information supplied by now inactive sources.  As an application, we have worked out the price of a vanilla option under the modulated information system, and we have highlighted the similarity of its functional form with that obtained by \cite{19}. Our second application addresses information asymmetry between agents whose access to information sources differ.  
We have concluded by producing an example of modulated information and the associated endogenous jump-diffusion dynamics. In the figure below, we show a simulated path of the conditional expectation process $(X_t)$ given in Proposition \ref{ExpectationSDE}. This jump-diffusion process jumps whenever an information coordinate is activated. At each point in time a Markov chain dictates which information coordinates are active.  
Because of the permanent nature of each ``shock'' given to $(X_t)$, whenever an information source is activated, we speak of a {\it permanent impact} on the value of $(X_t)$. The inclusion of a temporary impact on $(X_t)$ is a suggestion for future research, where the informational contribution of a particular information coordinate (source) ``fades away'' over time once the source is deactivated.
\begin{figure}[H]
\begin{center}
	\includegraphics[scale=0.55]{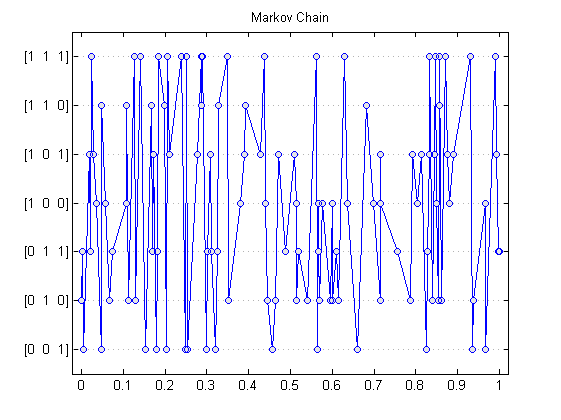}
	\includegraphics[scale=0.55]{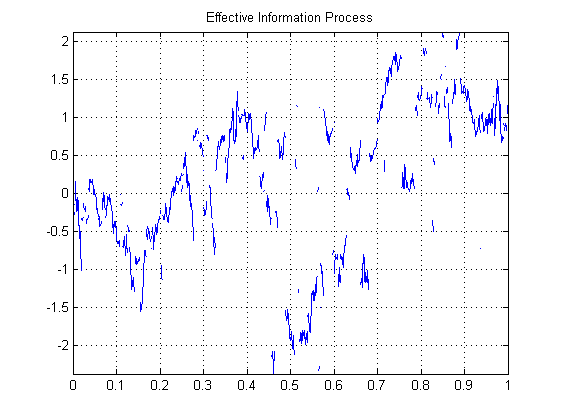}
	\\
	\includegraphics[scale=0.55]{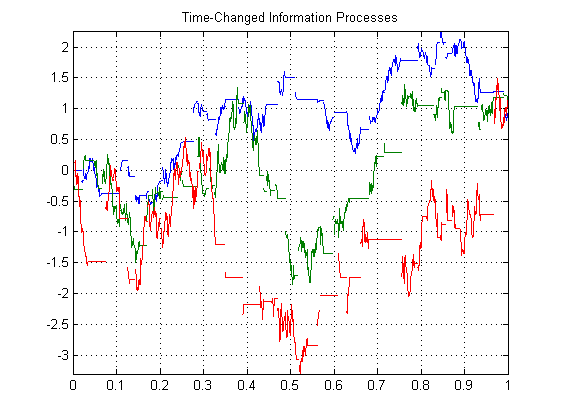}
	\includegraphics[scale=0.55]{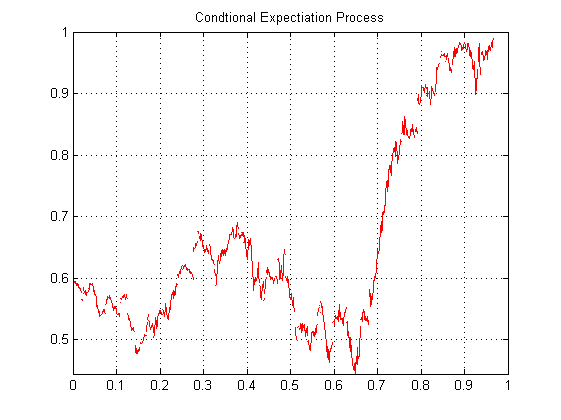}
	\caption{Simulation of endogenous jump-diffusion system (one path). Here, an example is shown whereby the random point field is illustrated by a Markov chain. Anti-clockwise, (a) the Markov chain, (b) the time-changed---or modulated---information process, (c) the effective information process, and (d) the endogenised jump-diffusion $X_t=\E[X_T\,\vert\,\F_t]$, for $0\le t <1$.}
\end{center}	
\end{figure}
\vspace{-1cm}
\section*{Acknowledgments}
The authors thank J. Akahori, G. W. Peters, J. Sekine, peer-reviewers, and participants in the Osaka-UCL Workshop on Stochastics, Numerics and Risk (March 2017), which was co-sponsored by The Daiwa Anglo-Japanese Foundation, participants in the 6th International Conference, Mathematics in Finance, RSA (August 2017), and in the Seminar of the Actuarial Mathematics \& Statistics Department, Heriot-Watt University (May 2019), for comments and suggestions. 

\end{document}